\documentclass[12pt]{amsart}
\usepackage{amsmath,amssymb,amsthm,amscd,stmaryrd,enumitem}
\usepackage{color}
\usepackage[mathscr]{euscript}
\usepackage[arrow,curve,frame,color]{xy}

\usepackage[usenames,dvipsnames]{xcolor}

\definecolor{gray}{gray}{0.7}

\newcommand{\slc}[1]{\langle #1\rangle} 
\newcommand{\cs}{\text{\rm c}} 
\newcommand{\olX}{\,\overline{\!X}}
\newcommand{\Ray}{{\rm R}}

\usepackage{hyperref}
\hypersetup{breaklinks=true}
\hypersetup{bookmarksdepth=3}

\usepackage[small,nohug,heads=vee]{diagrams}

\usepackage{graphicx}
\usepackage[labelformat=empty]{caption}
\usepackage{accents}

\usepackage[utf8]{inputenc} 
\usepackage[T1]{fontenc}
\usepackage{enumitem}

\usepackage{nicefrac}

\numberwithin{equation}{section}
\theoremstyle{plain}

\newtheorem{theorem}[equation]{Theorem}
\newtheorem{thm}[equation]{Theorem}
\newtheorem{proposition}[equation]{Proposition}
\newtheorem{lemma}[equation]{Lemma}
\newtheorem{corollary}[equation]{Corollary}
\newtheorem{conjecture}[equation]{Conjecture}

\newtheorem{lem}[equation]{Lemma}

\newtheorem{prop}[equation]{Proposition}
\newtheorem{cor}[equation]{Corollary}

\newtheorem{example}[equation]{Example}
\newcommand{\sub}{\subset}
\newcommand{\bZ}{{\mathbf Z}}

\theoremstyle{remark}
\newtheorem{remark}[equation]{Remark}

\theoremstyle{definition}
\newtheorem{definition}[equation]{Definition}

\newtheorem*{question*}{Question}

\newcommand{\T}{\text{\rm T}}
\newcommand{\CT}{\cC_\T} 
\newcommand{\dT}{d_\T} 

\newcommand{\bb}[1]{\llbracket #1\rrbracket}

\newcommand{\cC}{{\mathscr C}}

\newcommand{\E}{\mathbb E}
\newcommand{\F}{{\mathcal F}}

\newcommand{\I}{{\textbf I}}

\newcommand{\M}{{\mathbf M}}

\newcommand{\R}{\mathbb R}

\newcommand{\Z}{\mathbb Z}

\newcommand{\bI}{{\mathbf I}}

\newcommand{\al}{\alpha}

\newcommand{\ben}{\begin{enumerate}}
\newcommand{\bit}{\begin{itemize}}

\newcommand{\cat}{\operatorname{CAT}}

\newcommand{\cone}{\operatorname{C}}
\newcommand{\ccone}{\overline{\operatorname{C}}}
\newcommand{\D}{\partial}
\newcommand{\de}{\delta}

\newcommand{\een}{\end{enumerate}}
\newcommand{\eit}{\end{itemize}}
\newcommand{\eps}{\epsilon}

\newcommand{\Fill}{\operatorname{Fill}}

\newcommand{\ga}{\gamma}

\newcommand{\id}{\operatorname{id}}
\newcommand{\im}{\operatorname{Im}}

\newcommand{\la}{\lambda}

\newcommand{\om}{\omega}

\newcommand{\on}{\:\mbox{\rule{0.1ex}{1.2ex}\rule{1.1ex}{0.1ex}}\:}

\newcommand{\si}{\sigma}
\newcommand{\Si}{\Sigma}
\newcommand{\spt}{\operatorname{spt}}

\newcommand{\st}{\textrm{st}}

\newcommand{\ul}{\underline}
\newcommand{\loc}{\text{\rm loc}}
\newcommand{\es}{\emptyset}

\newcommand{\di}{\D_\infty}

\newcommand{\B}[2]{B_{#1}(#2)}   
\newcommand{\Sph}[2]{S_{#1}(#2)} 
\begin{document}

\title[Morse quasiflats II]{Morse Quasiflats II}
\author{Jingyin Huang}
\author{Bruce Kleiner}
\author{Stephan Stadler}
\thanks{The first author thanks Max-Planck Institute for mathematics at Bonn, where part of work was done. The second author was supported by NSF grant DMS-1711556 and a Simons Collaboration grant
a NSF grant DMS-1405899 and a Simons Fellowship. The third author was supported by  DFG grant SPP 2026.}

\date{\today}

\begin{abstract}
This is the second in a two part series of papers concerning Morse quasiflats -- higher dimensional analogs of Morse quasigeodesics.  Our focus here is on their asymptotic structure. In metric spaces with convex geodesic bicombings, we prove asymptotic conicality, uniqueness of tangent cones at infinity and Euclidean volume growth rigidity for Morse quasiflats.  Moreover, we provide some immediate consequences.

\end{abstract}
\maketitle

\tableofcontents

\section{Introduction}

\subsection{Overview} Gromov hyperbolicity has been a central concept in geometric group theory since it was first introduced in [Gro87]. Over the years, it has inspired a large number of variations,  which extend different aspects of hyperbolicity to more general settings \cite{farb1998relatively,bowditch2012relatively,dructu2005tree,osin2006relatively,Ballmann_axial,dructu2010divergence,ol2009lacunary,sela1997acylindrical,osin2016acylindrically,bestvina2015constructing,dahmani2017hyperbolically,monod2006orbit,hamenstadt2008bounded,thom2009low,masur2000geometry,kim2014geometry,behrstock2017hierarchically,bowditch2013coarse,alonso1995semihyperbolic,duchin2012statistical,chatterjee2019average} (the list here is not intended to be complete).  One strand in this literature is concerned with  ``directional hyperbolicity'', an approach originating in Ballmann's work on rank $1$ geodesics; this is a robust notion variously characterized as rank 1/Morse/contracting/sublinear (quasi)geodesics and subsets, or via tree graded structure, \cite{Ballmann_axial,bestvina2002bounded,dructu2010divergence,charney2014contracting,MR3737283,sisto2018contracting,sublinear_boundary,durham2015convex,cordes2017stability,MR3690269,kapovich1998quasi,dructu2005tree}.

Our aim in this paper and our previous paper \cite{HKS} is to develop higher dimensional aspects of directional hyperbolicity via Morse quasiflats -- higher dimensional analogs of Morse quasigeodesics. 
While \cite{HKS}  was primarily concerned with examining different alternative definitions of Morse quasiflats, proving their equivalence and quasi-isometry invariance, our objective in this paper is to establish asymptotic structural results.  
The present paper is independent from \cite{HKS}, apart from  a single statement which is rather intuitive (see Section~\ref{sec_definition}).

The main result here is  ``asymptotic conicality'' of Morse quasiflats: any sequence of blow-downs converges to a cone.
Hence, in analytic terms, Morse quasiflats possess unique tangent cones at infinity. The issue of uniqueness of tangent cones (at infinity) is 
fundamental in geometric analysis and has arisen in many places, in particular for (quasi)minimizing varieties \cite{T_reg, AA_rad,  Si_asym, W_tangent,RT_sing,higherrank}, harmonic maps \cite{ GW_rate, W_nonuni,Ha_sing}, Einstein manifolds \cite{ CT_cone,CM_uni_E}, and geometric flows \cite{W_size, W_nature,GK_uni, CM_uni_L,CM_level}.
While the uniqueness of (local) tangent cones is intimately related to regularity questions and the fine structure of singular sets, uniqueness of tangent cones
at infinity provide a description of the asymptotic structure --- essential to large-scale geometry.   The proofs of the above results have the same general strategy --  induction on scales -- but otherwise they vary considerably and are quite different from the argument used in this paper. 

\subsection{Motivation from large scale geometry}

For a space or a group $X$ satisfying some weak form of non-positive curvature condition, there is typically a space $\Sigma_X$ encoding the asymptotic intersecting pattern of certain collections of flats or abelian subgroups in $X$. This plays a fundamental role in understanding  the coarse geometry of $X$. Some well-known examples are:
\begin{enumerate}
	\item When $X$ is Gromov hyperbolic, $\Sigma_X$ is the Gromov boundary.
	\item When $X$ is a symmetric space of non-compact type or a Euclidean building, $\Sigma_X$ is the Tits boundary.
	\item When $X$ is a mapping class group, $\Sigma_X$ is the curve complex.
\end{enumerate}
In all these examples, $\Sigma_X$ concerns only top rank flats and their coarse intersections, which offers sufficiently robust information on their asymptotic geometry. However, for many other examples (e.g. Coxeter groups, Artin groups or more general $\cat(0)$ groups), the natural definition of $\Sigma_X$ necessarily involves flats/quasiflats (or abelian subgroups) which do not arise as coarse intersections of top rank flats, in order to avoid substantial loss of information \cite{kim2014geometry,morris2019parabolic,davis2017determining}. The task of identifying more general classes of flat/quasiflats that are coarse invariants is closely related to higher dimensional versions of Gromov hyperbolicity. This leads to the study of Morse quasiflats.

In examples (1)-(3) above, $\Sigma_X$ serves as a fundamental invariant in the study of quasi-isometric rigidity; a major step in proving the quasi-isometry invariance of $\Sigma_X$  is to understand the structure of top dimensional quasiflats \cite{Gromov_hyp,kleiner1997rigidity,eskin1997quasi,hamenstaedt2005geometry,MR2928983,bowditch2017large,behrstock2017quasiflats,fisher2018quasi}. Analogously, in more general situations one would start with an analysis of Morse quasiflats. The asymptotic conicality makes Morse quasiflats an accessible quasi-isometry invariant.

It is worth noting that certain lower dimensional quasiflats/flats have been studied earlier in different contexts, including relative hyperbolic spaces \cite{hruska_isolated,behrstock2009thick}, quasi-isometric classification of right-angled Artin groups and hierarchically hyperbolic spaces \cite{huang2017quasi,behrstock2017quasiflats}.  In fact, the lower dimensional quasiflats studied in these cases are specific examples of Morse quasiflats.

\bigskip
\subsection{The definition of Morse quasiflats}~
\label{subsec_intro_def}
We now give one of the definitions of Morse quasiflat from \cite{HKS}. Recall that each point in any asymptotic cone of a Morse quasi-geodesic  is a cut point of the cone \cite{dructu2010divergence}. A Morse quasiflat in our sense can be defined through a higher dimensional version of  this cut point property. 
%
%
%

\begin{definition}[Morse quasiflat]
	\label{def_morse_intro}
	An $n$-dimensional quasiflat $Q$ in a metric space $X$ is called {\em Morse}, if for any asymptotic cone $X_\om$ of $X$ and any $q_\om$ in the limit $Q_\om$ of $Q$, the map $ H_n(Q_\om,Q_\om\setminus\{q_\om\},\mathbb Z)\to H_n(X,X\setminus\{q_\om\},\mathbb Z)$ is injective.
\end{definition}
See Section~\ref{sec_definition}, as well as  \cite{HKS}, for several equivalent definitions without using asymptotic cones.

\bigskip
\begin{remark}
In recent literature the term ``Morse''  was used in different contexts:
\begin{enumerate}
	\item Morse subsets \cite{MR3690269,genevois2017hyperbolicities} respectively strongly quasiconvex subsets \cite{Tran_strongly_quasiconvex}, which can be viewed as a local version of relative hyperbolicity;

\item a Morse lemma was proved for regular quasi-geodesics in \em{higher rank} symmetric spaces \cite{KLP_morse}.
\end{enumerate}
We emphasize that although the historical origin of ``Morse'' in (1) and (2) is the same as for our work, we caution the reader that meanings are usually not compatible.
\end{remark}

For instance a Morse quasi-geodesic $l$ in a finitely generated group $G$ 
gives a Morse quasiflat $l\times l$ in the product $G\times G$. Such a quasiflat will not be a Morse subset in the sense of \cite{MR3690269,genevois2017hyperbolicities}, unless $G$ is virtually $\Z$. Another instructive example to keep in mind is that a periodic flat in a proper CAT(0) space is a 
Morse quasiflat if and only if it does not bound a flat half-space. 
Morse quasiflats are in generally not isolated, i.e. they usually intersect other quasiflats along non-trivial sub-quasiflats.
 More interesting examples can be found in \cite[Section 1.6]{HKS}.
 
We also introduce the notion of pointed Morse quasiflat, which is identical to Definition~\ref{def_morse_intro} except that the basepoint $q_\om$ comes from the constant sequence (see Definition~\ref{def:pointed}). Roughly speaking, a pointed Morse quasiflat is allowed to be less and less ``Morse'' if we move further and further away from the base point. Morse quasiflats are invariant under quasi-isometries, while pointed Morse quasiflats are even invariant under sublinearly bilipschitz equivalences in the sense of \cite{cornulier2019sublinear}.

\bigskip
\subsection{Statement of results}
  \label{subsec_visibility}
 
For simplicity, we will state the results for $\cat(0)$ spaces. The main theorem (Theorem~\ref{thm_sublinear_close_intro}) is proved in the more general setting of spaces with a convex geodesic bicombing (see Definition~\ref{def:bicombing}), 
including Busemann convex spaces and injective metric spaces. Let $X$ be a $\cat(0)$ space. 
Let $\D_T X$ be the Tits boundary of $X$. For a base point $p\in X$ and a subset $A\subset \D_TX$, we denote the union of all geodesics from $p$ to a point in $A$ by $\cone_p(A)$. This is called a \emph{geodesic cone} over $A$. We use $B_p(r)$ to denote the ball of radius $r$ at $p$.

Recall that quasi-geodesics in Gromov-hyperbolic spaces are at finite Hausdorff distance from geodesics. While quasi-geodesics in $\cat(0)$ spaces do not enjoy such a property in general (e.g. the logarithmic spiral quasi-geodesic in $\R^2$), it was known that top-dimensional quasiflats in $\cat(0)$ spaces are relatively well-behaved \cite{higherrank}, notably for their ``cone-like'' feature. For instances, top-dimensional quasiflats in symmetric spaces of non-compact type are Hausdorff close to a finite union of Weyl cones \cite{kleiner1997rigidity,eskin1997quasi}, and top-dimensional quasiflats in $\cat(0)$ cube complexes are Hausdorff close to a finite union of orthants \cite{bks1,huang_quasiflat,behrstock2017quasiflats,bowditch2019quasiflats}. Note that Weyl cones and orthants are specific types of geodesic cones, and top-dimensional quasiflats are specific types of Morse quasiflats.  Our main result shows that Morse quasiflats in general $\cat(0)$ spaces are sublinearly close to geodesic cones.
 

\begin{thm}[Asymptotic conicality,  Corollary~\ref{cor_sublinear_close}]
	\label{thm_sublinear_close_intro}
	\quad Suppose $Q$ is a Morse $(L,A)$-quasiflat in a proper $\cat(0)$ space $X$. Then there is a subset $\D_T Q\subset\D_T X$ such that for any base point $p\in X$ the following holds true :  
	\begin{equation*}
	\lim_{r\to\infty} \frac{d_H(B_p(r)\cap Q,B_p(r)\cap C_p(\partial _T Q))}{r}=0.
	\end{equation*}
In particular, the subset $\partial_T Q$ consists of ideal points  represented by rays which are sublinearly close to $Q$ in the sense of Definition~\ref{def_Tits_boundary}.
\end{thm}

We draw the reader's attention to a variety of enhancements of the theorem that may be found in Section~\ref{sec_visibility_morse_quasiflats}, such as Theorem~\ref{thm_cone_is_close_to_quasiflat} (another form of asymptotic conicality), Proposition~\ref{prop_unique_tangent_cone_infty} (uniqueness of tangent cone at infinity) and Proposition~\ref{prop:full support} (structure of $\partial_T Q$). 

It is natural to ask when the sublinear estimate of Theorem~\ref{thm_sublinear_close_intro} can be improved to a finite Hausdorff distance estimate. While the stronger estimate fails in general, we provide a simple criterion showing that it does hold in many interesting cases (see Example~\ref{example_exp} and Proposition~\ref{prop_hausdorff}).
  
The asymptotic conicality exhibited in Theorem~\ref{thm_sublinear_close_intro} has  precursors.  The first is Eberlein and O'Neill's notion of ``visibility''   \cite{eberlein1973visibility}. The $n=1$ case has been known for a while \cite[Proposition 3.24]{dructu2010divergence}. Also, the special case of top-dimensional quasiflats was covered more recently by one of the main results from \cite{higherrank}.  

The proof of Theorem~\ref{thm_sublinear_close_intro}, like that of \cite{higherrank} and the other results on uniqueness of tangent cones mentioned in the overview above, is based on an induction on scales.  However the argument here is quite different; in particular,   the approach in \cite{higherrank} breaks down completely without the assumption of top dimensionality.     See Section~\ref{sec_informal_proof} for more on this, and a sketch of the proof.

Let $\di Q$ be the set of limit points of $Q$ in the ideal boundary $\di X$ with respect to the cone topology, then we have $\di Q=\D_T Q$ (see Lemma~\ref{lem_ideal_equ_tits}); though $\partial_\infty X$ is not a quasi-isometry invariant in general \cite{croke2000spaces}, this shows that quasi-isometries actually respect subsets of $\partial_\infty X$ arising from Morse quasiflats (see Corollary~\ref{cor_boundary_map}).

\bigskip
\begin{remark}
It is natural to ask whether the subset $\partial_T Q$, endowed with induced metric from $\D_T X$, is bilipschitz to a standard sphere. This seems to be a rather subtle issue. However, we know the Euclidean cone over $\partial_T Q$ is bilipschitz homeomorphic to $\mathbb E^n$, see Proposition~\ref{prop:full support}; moreover $\partial_T Q$ has the structure of a cycle in the sense of Definition~\ref{morse} and Proposition~\ref{prop:full support}. Also Remark~\ref{rmk:sphere} gives cases when $\partial_T Q$ is indeed a sphere, which applies to Theorem~\ref{thm:cube Morse lemma}.
\end{remark}

%

\begin{remark}
	\label{rmk:cycle}

In Theorem~\ref{thm_sublinear_close_intro}, it is crucial that quasiflats ``do not have boundary'' (they can be represented by Lipschitz chains, and as such they are cycles).
For instance, the conclusion of Theorem~\ref{thm_sublinear_close_intro} is not true for quasi-isometrically embedded Euclidean sectors which are Morse (consider 
the image of a quadrant in the Euclidean plane under a self-quasi-isometry).
\end{remark}

We also provide uniqueness and rigidity results.

\begin{thm}[Theorem~\ref{thm:uniqueness body}]
	\label{thm:uniqueness}
	Suppose $X$ is a proper $\cat(0)$ space, and $Q_1,\,Q_2\subset X$ are Morse quasiflats. Then there exists a positive constant $C$, depending only on $X,\dim Q_1$, the quasi-isometry constants of $Q_1$ and the Morse data of $Q_1$, such that
	$\partial_T Q_1=\partial_T Q_2$ implies
	\[d_H(Q_1,Q_2)<C.\]
\end{thm}

\begin{thm}[Theorem~\ref{thm_growth_rigidity body}]\label{thm_growth_rigidity}
 Let $X$ be a proper CAT(0) space. Let $Q\subset X$ be an $n$-dimensional Morse quasiflat.
Suppose that the volume growth of $Q$ is at most Euclidean.
 Then there is an $n$-flat $F\subset X$ with $d_H(F,Q)<C$ where $C$ depends only on $n,X$ and the Morse data of $Q$.
\end{thm}

\bigskip
\subsection{ Immediate consequences and further discussion} 
In this subsection we point out some settings where Morse quasiflats arise naturally and one may apply our main theorem to produce new quasi-isometry invariants.   The resulting information may potentially be used to deduce quasi-isometric rigidity and classification results; as the methods for doing this are usually more combinatorial in flavor and strongly depend on the specific setting, we will not treat this aspect here, except one case which we postpone to the appendix (as the main focus of the paper is on the metric aspects of Morse quasiflats).

Theorem~\ref{thm_sublinear_close_intro} and Theorem~\ref{thm:uniqueness} reduce the study of Morse quasiflats to the study of certain cycles in the Tits boundary, which we call immovable cycles as in Definition~\ref{morse}. Once the combinatorial structure of these cycles is understood, 
one obtains structural results of Morse quasiflats in the space. This gives combinatorial invariants for quasi-isometries as Morse quasiflats are quasi-isometry invariants. 

For example, for a Morse quasiflat $Q$ in a CAT(0) cube complex, the immovable cycle $\partial_T Q$ can be ``filled'' by another quasiflat $Q'$ which is a union of orthants in the sense of \cite[Theorem 1.4]{boundary} \footnote{A more general version of \cite[Theorem 1.4]{boundary} has been obtained recently by \cite{FFH}}. By Theorem~\ref{thm_sublinear_close_intro}, $Q'$ and $Q$ are sublinearly close. Hence $Q$ and $Q'$ have finite Hausdorff distance by \cite[Proposition 10.4]{HKS}, which implies the following.

\begin{theorem}
	\label{thm:cube Morse lemma}
Suppose $X$ is a finite dimensional proper $\cat(0)$ cube complex. If $Q\subset X$ is a $k$-dimensional Morse quasiflat, then there exists a collection of pairwise disjoint $k$-dimensional $\cat(0)$ orthants 
$\{O_i\}_{i=1}^k$ such that $d_H(Q,\sqcup_{i=1}^k O_i)<\infty$.

If $Q$ is pointed Morse, then there exists a collection of pairwise disjoint $k$-dimensional $\cat(0)$ orthants $\{O_i\}_{i=1}^k$ such that $Q$ and $\sqcup_{i=1}^k O_i$ are sublinearly close in the sense of 
Theorem~\ref{thm_sublinear_close_intro}.

Moreover, in each of the above cases, the $\cat(0)$ orthants are at finite Hausdorff distance from some $\ell^1$-orthants.  
\end{theorem}

An $\ell^1$-orthant is a subcomplex (with induced $\ell^1$ metric from $X$) isometric to a standard Euclidean orthant equipped with the $\ell^1$-distance through a cubical isomorphism.	
In both cases of Theorem~\ref{thm:cube Morse lemma}, each $O_i$ is contained in a convex subcomplex $O'_i$ of $X$ such that $O'_i$ splits as a product of $k$ cubical factors ($k=\dim Q$). 
In the first case of Theorem~\ref{thm:cube Morse lemma}, $O'_i$ and $O_i$ have finite Hausdorff distance, in the second case $O'_i$ and $O_i$ are sublinearly close.

		Special cases of Theorem~\ref{thm:cube Morse lemma} for top dimensional quasiflats were obtained in \cite{bks1, huang_quasiflat,behrstock2017quasiflats,bowditch2019quasiflats} by different methods. $\cat(0)$ cube complexes typically contain plenty of Morse quasiflats which are neither 1-dimensional nor of top rank.

Combining Theorem~\ref{thm:cube Morse lemma} with the argument in \cite[Section 5]{huang_quasiflat}, we obtain the following ``Morse lemma'' for Morse \emph{flats}, which gives new quasi-isometric invariants for virtually compact special groups.

\begin{thm}
	\label{thm:morse lemma}
Let $ X_1$ and $ X_2$ be universal covers of compact special cube complexes $\bar X_1$, $\bar X_2$, respectively. If $q:X_1\to X_2$ is an $(L,A)$-quasi-isometry, then for any Morse flat $F_1\subset X_1$, there exists a Morse flat $F_2\subset X_2$ such that $d_H(q(F_1),F_2)<C$ where $C<\infty$ depends only on $X_1,X_2,L,A$ and the Morse data of $F_1$.
\end{thm}

In the setting of the theorem, if $A\subset \pi_1(\bar X_1)$ is a free abelian subgroup not virtually contained in a higher rank free abelian subgroup, then $A$ is Morse \cite[Corollary 1.20]{HKS}, hence $q(A)$ is at finite Hausdorff distance from a flat by Theorem~\ref{thm:morse lemma}. This also implies for \emph{any} abelian subgroup $A'\subset \pi_1(\bar X_1)$, the image $q(A')$ is contained in a finite neighborhood of a flat.

Now we mention one consequence of Theorem~\ref{thm:morse lemma}. Recall from \cite{bks2} that a simplicial graph $\Gamma$ is \emph{atomic} if $\Gamma$ is connected, does not have $n$-cycle with $n<5$ and does not contain any vertex $v$ such that $\st(v)$ separates $\Gamma$. The main result of \cite{bks2} may be rephrased as the assertion that two graph products of $\mathbb Z$s with atomic defining graphs are quasi-isometric if and only if the underlying graphs are isomorphic.  We now prove the analogous assertion for graph products of arbitrary rank one right-angled Coxeter groups. As different vertex groups of the graph products might contain flats of different dimension, the usual strategy of controlling quasi-isometric images of top-dimensional quasiflats is less effective. Instead we use 2-dimensional Morse flats coming from products of rank one axes in the vertex groups.

\begin{corollary}[Corollary~\ref{cor_qi_classification}]
	\label{cor_qi_classification_intro}
	Let $\Gamma_1$ and $\Gamma_2$ be atomic graphs. Let $G_{\Gamma_1}$ and $H_{\Gamma_2}$ be two graph products with vertex groups being rank one right-angled Coxeter groups. Then $G_{\Gamma_1}$ and $H_{\Gamma_2}$ are quasi-isometric if and only if there exists a graph isomorphism $f:\Gamma_1\to \Gamma_2$ such that $G_{v}$ is quasi-isometric to $H_{f(v)}$ for any $v\in\Gamma_1$.
\end{corollary}

The assumptions here are not intended to be optimal -- we only present a relatively simple case to illustrate the idea.  We expect that the rigidity assertion in  Corollary~\ref{cor_qi_classification_intro} holds in greater generality, see Remark~\ref{rmk_graph_product}.

We end this section with a few more natural speculations. The following is an analog of Theorem~\ref{thm:cube Morse lemma} and Theorem~\ref{thm:morse lemma} for Coxeter groups.

\begin{conjecture}
Morse quasiflats in the Davis complexes of Coxeter groups are at finite Hausdorff distance from a union of $\cat(0)$ orthants in the sense of Theorem~\ref{thm:cube Morse lemma}. 
Moreover, Theorem~\ref{thm:morse lemma} holds when $X$ and $Y$ are the Davis complexes of some Coxeter groups.
\end{conjecture}

Another potentially interesting case is provided by Artin groups of type FC. They act geometrically on injective metric spaces \cite{huang2019helly}, 
hence Corollary~\ref{cor_sublinear_close} and Theorem~\ref{thm:uniqueness body} apply. Moreover, they contain plenty of Morse quasiflats.

\bigskip
\subsection{Structure of the paper}
Section~\ref{sec_prelim} - Section~\ref{sec_definition} are preparatory in nature. 
In Section~\ref{sec_prelim} we discuss some background on metric spaces and metric currents and agree on notation. 
In Section~\ref{sec:quasiflats in metric spaces} we prove some properties of quasiflats specific to metric spaces with convex geodesic bicombings, including the representability by Lipschitz quasiflats and 
the existence of Lipschitz retractions.

In Section~\ref{sec_informal_proof} we give an informal discussion on the properties of Morse quasiflats as well as the proof of the main result Theorem~\ref{thm_sublinear_close_intro}. 
In Section~\ref{sec_definition}, we provide precise definitions of Morse quasiflats and recall several essential features like the coarse neck property and the coarse piece property.

In Section~\ref{sec_visibility_morse_quasiflats} we prove our main structural results: asymptotic conicality, visibility and uniqueness of tangent cones at infinity for Morse quasiflats. In Section~\ref{sec_rigidity}, we prove a rigidity result for Morse quasiflats with Euclidean mass growth. 

In the appendix Section~\ref{sec_appendix}, we use Theorem~\ref{thm_sublinear_close_intro} to exhibit some examples of quasi-isometric classification.

\subsection{Acknowledgements}
We would like to thank Sam Shepherd for comments on an earlier version of this paper. We thank Bernhard Leeb for asking a question which led to Example~\ref{example_exp}. We also thank the anonymous referee for valuable comments.

\section{Preliminaries}
\label{sec_prelim}

\subsection{Metric notions} \label{subsect:metric}
\mbox{}
We will denote by $\mathbb E^n$ the $n$-dimensional Euclidean space with its flat metric and by $\mathbb S^{n-1}$ the $(n-1)$-dimensional round unit sphere.

Let $X = (X,d)$ be a metric space. For $\lambda>0$, we denote the rescaled space $(X,\lambda\cdot d)$ simply by $\lambda\cdot X$. 
We write 
\[
\B{p}{r} := \{x \in X : d(p,x) \le r\}, \quad
\Sph{p}{r} := \{x \in X : d(p,x) = r\}
\] 
for the closed ball and sphere with radius $r \ge 0$ and center $p \in X$.

A map $f \colon X \to Y$ into another metric space $Y = (Y,d)$ is
{\em $L$-Lipschitz}, for a constant $L \ge 0$, 
if for all $x,x' \in X$ holds
\[d(f(x),f(x')) \le L\cdot d(x,x').\]

A map $f \colon X \to Y$ between two metric spaces is called an
{\em $(L,A)$-quasi-isometric embedding}, for constants
$L \ge 1$ and $A \ge 0$, if
\[
\frac{1}{L}\cdot d(x,x') - A \le d(f(x),f(x')) \le L\cdot d(x,x') + A
\]
for all $x,x' \in X$.
A {\em quasi-isometry} $f \colon X \to Y$ has the additional property
that $Y$ is within finite distance of the image of $f$. An \emph{$(L,A)$-quasi-disk} $D$ in a metric space $X$ is the image of an $(L,A)$ quasi-isometric embedding 
$\Phi$ from a closed metric ball $B$ in $\R^n$ to $X$.
The \emph{boundary} of $D$, denoted $\D D$, 
is defined to be $\Phi(\D B)$. An $n$-dimensional {\em quasiflat} in $X$ is the image of a quasi-isometric
embedding of $\R^n$.

A curve $\rho \colon I \to X$ defined on some interval $I \sub \R$ is a
{\em geodesic} if there is a constant $s \ge 0$, the {\em speed}
of $\rho$, such that $d(\rho(t),\rho(t')) = s |t - t'|$ for all $t,t' \in I$.
A geodesic defined on $I = \R_+ := [0,\infty)$ is called a {\em ray}.

\subsection{Metric spaces with convex geodesic bicombing}
\label{subsec:bicombing}
\begin{definition}[convex bicombing] \label{def:bicombing}
	By a {\em convex bicombing} $\si$ on a metric space $X$ 
	we mean a map $\si \colon X \times X \times [0,1] \to X$ such that
	\ben
	\item   
	$\si_{xy} := \si(x,y,\cdot) \colon [0,1] \to X$ is a geodesic from
	$x$ to $y$ for all $x,y \in X$;
	\item
	$t \mapsto d(\si_{xy}(t),\si_{x'y'}(t))$ is convex on $[0,1]$
	for all $x,y,x',y' \in X$;
	\item
	$\im(\si_{pq}) \sub \im(\si_{xy})$ whenever $x,y \in X$ and
	$p,q \in \im(\si_{xy})$.
	\een
	A geodesic $\rho \colon I \to X$ is then called a
	{\em $\si$-geodesic} if $\im(\si_{xy}) \sub \im(\rho)$ whenever
	$x,y \in \im(\rho)$.
	A convex bicombing $\si$ on $X$ is {\em equivariant\/} if
	$\ga \circ \si_{xy} = \si_{\ga(x)\ga(y)}$ for every
	isometry $\ga$ of $X$ and for all $x,y \in X$.
\end{definition}

Note that in~(3), we do not specify the order of $p$ and $q$ with respect to
the parameter of $\si_{xy}$, in particular $\si_{yx}(t) = \si_{xy}(1-t)$.
In the terminology of~\cite{DesL1}, $\si$ is a {\em reversible} and
{\em consistent convex geodesic bicombing} on $X$.
In Section~10.1 of~\cite{kleiner1999local}, metric spaces with such a structure $\si$
are called {\em often convex}. 
This class of spaces includes all CAT(0) spaces, Busemann spaces,  
as well as (linearly) convex subsets of normed spaces; at the same time, 
it is closed under various limit and product constructions such as 
ultralimits, (complete) Gromov--Hausdorff limits, and $l_p$ products 
for $p \in [1,\infty]$.

Let $X$ be a complete metric space with a convex bicombing $\si$.
The boundary at infinity of $(X,\si)$ is defined in the usual way, as for
CAT(0) spaces, except that only $\si$-rays are taken into account.
Specifically, we let $\Ray^\si X$ and $\Ray^\si_1 X$ denote the sets of
all $\si$-rays and $\si$-rays of speed one, respectively, in $X$.
For every pair of rays $\rho,\rho' \in \Ray^\si X$, the function
$t \mapsto d(\rho(t),\rho'(t))$ is convex, and $\rho$ and $\rho'$
are called {\em asymptotic} if this function is bounded.
This defines an equivalence relation $\sim$ on $\Ray^\si X$ as well as
on $\Ray^\si_1 X$.
The {\em boundary at infinity} or {\em visual boundary} of $(X,\si)$ is 
the set
\[
\di X := \Ray^\si_1 X/{\sim}
\]
Given $\rho \in \Ray^\si_1 X$ and $p \in X$, there is a unique ray 
$\rho_p \in \Ray^\si_1 X$ asymptotic to $\rho$ with $\rho_p(0) = p$. The set
\[
\olX := X \cup \di X
\]
carries a natural metrizable topology,
analogous to the cone topology for CAT(0) spaces. With this topology,
$\olX$ is a compact absolute retract, and $\di X$ is a $Z$-set
in $\olX$. See Section~5 in~\cite{DesL1} for details.
For a subset $A \sub X$, the \emph{ideal boundary} of $A$, denoted by $\di A$, is defined as the set of
all points in $\di X$ that belong to the closure of $A$ in~$\olX$.
For a point $p \in X$ we define the geodesic homotopy
\[
h_p \colon [0,1] \times X \to X
\]
by $h_p(\la,x) := h_{p,\la}(x) := \si_{px}(\la)$. Note that
the map $h_{p,\la} \colon X \to X$ is $\la$-Lipschitz. 
For a set $A \sub X$,
\[
\cone_p(A) := h_p([0,1] \times A)
\]
denotes the geodesic cone from $p \in X$ over $A$, and $\ccone_p(A)$
denotes its closure in $X$. Similarly, if $\Lambda \sub \di X$, then 
$\cone_p(\Lambda) \sub X$ denotes the union of the traces of the rays emanating 
from $p$ and representing points of $\Lambda$.

The \emph{Tits cone} of $(X,\si)$ is defined as the set 
\[
\CT X := \Ray^\si X/{\sim},
\]
equipped with the metric given by 
\[
\dT([\rho],[\rho']) := \lim_{t \to \infty} \frac{1}{t}\,d(\rho(t),\rho'(t)).
\]
Note that $t \mapsto d(\rho(t),\rho'(t))$ is convex,
thus $t \mapsto d(\rho(t),\rho'(t))/t$ is non-decreasing if $\rho,\rho'$
are chosen such that $\rho(0) = \rho'(0)$. From this it is easily 
seen that $\CT X$ is complete.

\begin{definition}
	\label{def_exp}
For each base point $p\in X$, we can define an exponential map $\exp_p:\CT X \to X$ sending the class $[\rho]\in \Ray^\si X$ to $\rho_p(1)$ where $\rho_p$ is the unique $\si$-ray in $\Ray^\si X$ asymptotic to $\rho$ with $\rho_p(0)=p$. Note that $\exp_p$ is 1-Lipschitz.
\end{definition}

The {\em Tits boundary} of $(X,\si)$ is the unit sphere
\[
\D_T X := \Sph{o}{1} = \Ray^\si_1 X/{\sim}
\]
in $\CT X$, endowed with the topology induced by $\dT$. This topology
is finer than the cone topology on the visual boundary $\di X$, 
which agrees with $\D_T X$ as a set. 

\begin{definition}
	\label{def_Tits_boundary}
For a subset $A\subset X$ we define the {\em Tits boundary} $\D_T A$ of $A$ as the collection points in $\partial_T X$ represented by geodesic ray $\rho$ such that there exists a sequence $(x_k)$ in $\rho$ such that $d(x_k,\rho(0))\to\infty$ and 
\[ \lim\limits_{k\to\infty}\frac{d(x_k,A)}{d(x_k,\rho(0))}=0.\]
\end{definition}

Note that for closed subsets $\Lambda\subset \di X$ holds $\di \cone_p(\Lambda)=\D_T \cone_p(\Lambda)=\Lambda$. Also $\partial_T A$ could possibly be empty even if $A$ is unbounded.

\subsection{Local currents in proper metric spaces} 
\label{subsect:currents}
\mbox{}

We will use the theory of (metric) integral currents throughout.
The reader will find an overview of what is needed in \cite{higherrank} and \cite{HKS} while we refer to \cite{ambrosio2000currents} and \cite{Lan3} for a thorough treatment.
Here we will only agree on notation.

We denote the space of {\em $n$-dimensional locally integral currents} by $\I_{n,loc}(X)$.
We write $\bI_{n,\cs}(X)$ (resp. $\bI_n(X)$) for the respective subgroups of
{\em integral currents\/} with compact support (resp. with finite mass). The corresponding subgroups of {\em cycles} are denoted by $\bZ_{n,loc}(X)$ and
$\bZ_{n,\cs}(X)$ (resp. $\bZ_n(X)$).
Let $T$ be a current on a proper metric space $X$. Then we denote by $\D T$ its {\em boundary}; by $\|T\|$ its associated Radon measure;  by $\M(T)=\|T\|(X)$ its {\em mass} and by $\spt(T)$ its support.
For a Lipschitz map $f:X\to Y$ to another proper metric space,  we denote by $f_\# T$ the {\em push-forward} of $T$ by $f$. A current $T'$ is called a {\em piece} of $T$ if
$\|T\|=\|T-T'\|+\|T'\|$ holds and the corresponding decomposition is called a {\em piece decomposition}. For a Borel subset $B\subset X$, let $T\on B$ be the \emph{restriction} of $T$ to $B$. Then $T\on B$ is a piece of $T$. 

If $\varphi:X\to\R$ is an $L$-Lipschitz function, then for almost every real number $s$  the {\em slice} of $T$ by $\varphi$ at $s$ is defined as
\[\slc{ T,\varphi, s}:=(\D T)\on\{\varphi>s\}-\D(T\on\{\varphi>s\}).\]
 We recall the {\em coarea inequality} which we will use intensively throughout.
For every Borel subset $B\subset\R$ holds
\[
\int_B \M(\slc{T,\varphi,s}) \,ds \le L\cdot\|S\|(\varphi^{-1}(B)).
\]

Recall that every function
$w \in L^1_\loc(\R^n)$ induces a current 
$\bb{w}$ defined by
\[
\bb{w}(\pi_0,\dots,\pi_n) 
:= \int_{\R^n} w\cdot \pi_0\cdot\det\bigl[\partial_j\pi_i\bigr]_{i,j = 1}^n \,dx
\]
For a characteristic function $\chi_W$ of a Borel set $W \sub \R^n$,
we put $\bb{W} := \bb{\chi_W}$.
(See Section~2 in~\cite{Lan3} for details.) 

If $T\in \bI_{N,\loc}(\R^N)$, then $T=\bb{u}$ for some function $u$ of locally bounded variation, moreover $u$ is integer-valued almost everywhere \cite[Theorem 7.2]{Lan3}. 
The element $\bb{\R^N}\in \bI_{N,\loc}(\R^N)$ is called \emph{fundamental class} of $\R^n$.

\begin{lemma}[coning inequality]\label{lem:coning_inequality}
Let  $X$ be a complete metric space with a convex geodesic bicombing. Then every cycle 
	$S \in \bZ_{n-1}(X)$ possesses for every point $p\in X$ a conical filling $\cone_p(S) \in \bI_n(X)$ with $\spt(\cone_p(S))\subset\cone_p(\spt(S))$ called {\em cone from p over $S$}. If $\spt(S)\subset B_p(R)$,
	then $\cone_p(S)$ fullfills the coning inequality
	\[\M(\cone_p(S))\leq R\cdot\M(S).\]
 
\end{lemma}

See \cite[Section 2.3]{wenger2005isoperimetric}.

\begin{theorem}[isoperimetric inequality] \label{thm:isop-ineq}
	Let $n \ge 2$, and let $X$ be a complete metric space with a convex geodesic bicombing. Then every cycle 
	$S \in \bZ_{n-1}(X)$ possesses a filling $T \in \bI_n(X)$  such that
\begin{enumerate}
\item 	$\M(T) \le c_0\cdot\M(S)^{n/(n-1)}$;
\item $\spt(T)\subset N_{c_1 \M(S)^{1/n-1}}(\spt(S))$,
\end{enumerate}
	for constants $c_0,c_1 > 0$ depending only on $n$. Moreover, if $S$ has compact support, then we can also require $T$ to have compact support.
\end{theorem}

The first item is proved in \cite{wenger2005isoperimetric}, see the comment after \cite[Theorem 1.2]{wenger2005isoperimetric} regarding compact supports. 
The second item follows from \cite[Proposition 4.3 and Corollary 4.4]{wenger2011compactness}.

\subsection{Minimizers and density}

Suppose $X$ is a complete metric space. We say an element $T\in \I_n(X)$ is \emph{minimizing}, or $T$ is a \emph{minimizer}, if $\M(T)\le \M(T')$ for any 
$T'\in\I_n(X)$ with $\D T=\D T'$. For a constant $\Lambda\ge 1$, we say $T$ is \emph{$\Lambda$-minimizing}, if for each piece $T'$ of $T$, we have $\M(T')\le \Lambda\cdot\M(T'')$ for any $T''\in\I_n(X)$ with $\D T''=\D T'$. 
Note that $T$ is minimizing if and only if $T$ is 1-minimizing. A local current $T\in \I_{n,\loc}(X)$ with $X$ being proper is \emph{minimizing}, if each compactly supported piece of $T$ is minimizing. 
We define $\Lambda$-minimizing for local currents in a similar way.

For $S\in \bZ_n(X)$, we define the {\em filling mass} by $\Fill(S):=\inf \{\M(T):T\in \I_{n+1}(X), \D T=S\}$. Further, we define the {\em filling distance} between cycles $S, S'\in \bZ_n(X)$ by
\[\F(S,S')=\Fill(S-S').\]

\begin{theorem} \label{thm:plateau}
	\cite[Theorem 2.4]{higherrank}
	Let $n \ge 1$, and let $X$ be a proper metric space with a convex geodesic bicombing. Then for every $S \in \bZ_{n-1,\cs}(X)$
	there exists a filling $T \in \bI_{n,\cs}(X)$ of $S$ with mass $\M(T)=\Fill(S)$.
	Furthermore, $\spt(T)$ is within distance at most $(\M(T)/\de_0)^{1/n}$
	from $\spt(S)$ for some constant $\de_0 > 0$ depending only on $n$.
\end{theorem}

Recall the following special case of
\cite[Definition 3.1]{higherrank}. A cycle $S \in \bZ_{n,\loc}(X)$ in a proper metric space is
called  {\em $(\Lambda,a)$-quasi-minimizing}, if for all $x\in\spt(S)$ and almost all $r>a$
holds
\[\M(S\on \B{x}{r})\leq\Lambda\cdot \M(T)\]
whenever $T\in \bI_{n+1,\cs}(X)$ with $\D T=\D(S\on \B{x}{r})$.

\begin{lemma}[density] \label{lem:density}
	Let $n \ge 1$, let $X$ be a proper metric space with a convex geodesic bicombing. 
	If $S \in \bZ_{n,\loc}(X)$ is $(\Lambda,a)$-quasi-minimizing,
	and if $x \in \spt(S)$ and $r > 2a$,
	then
	\[
	\|S\|(\B{x}{r}) \ge \theta_0\cdot r^n
	\] 
	for some constant $\theta_0 > 0$ depending only on $n$ and $\Lambda$.
\end{lemma}
This is a special case of \cite[Lemma 3.3]{higherrank}. 

\begin{lemma}\label{lem:fill-density}
	\cite[Lemma 3.4]{higherrank}
	Let $n \ge 1$, let $X$ be a complete metric space with a convex geodesic bicombing. 
	If $S \in \bZ_n(X)$ (or $S\in \bZ_{n,\loc}(X)$ when $X$ is proper) is $(\Lambda,a)$-quasi-minimizing,
	and if $x \in \spt(S)$ and $r > 4a$, 
	then 
	\[	
		\inf\{\M(V) : V \in \bI_{n+1,\cs}(X),\,\spt(S-\D V) \cap \B{x}{r} = \es\} \ge \theta_1\cdot r^{n+1}
	\]
	for some constant $\theta_1 > 0$ depending only on $n$, the constant $\theta_0$ from 
	Lemma~\ref{lem:density}, and $\Lambda$.
\end{lemma}

\section{Quasiflats in metric spaces}
\label{sec:quasiflats in metric spaces}
The main goal of this section is to provide some auxiliary results on building chains between cycles in the space and their ``projections'' on quasiflats.

\subsection{Lipschitz quasiflats and quasi-retractions}
\label{subsec:Lipschitz quasifalts}

We recall the following result which allows us to replace quasidisks with Lipschitz continuous quasidisks, at least in the presence of a convex geodesic bicombing.
It was proven in \cite[Lemma 1.2]{LS} for Hadamard spaces but the proof extends to our setting.

\begin{lemma}\label{lem_lip_quasidisc}
	Let $X$ be a metric space with a convex geodesic bicombing and let $\Phi:B\to X$ be an $n$-dimensional $(L,A)$-quasidisk. 
	Then there exist constants $L', A'$ depending only on $L, A, n$ and an $L'$-Lipschitz $(L',A')$-quasidisk $\Phi':B\to X$
	such that $d(\Phi(x), \Phi'(x))\leq A'$ for all $x\in B$.

\end{lemma}

From now on we will restrict our attention to $L$-Lipschitz  $(L,A)$-quasiflats/quasidisks.

\begin{definition}
	Let $K\subset X$ be a closed subset. A map $\pi:X\to K$ is called a $\lambda$-quasi-retraction if it is $\lambda$-Lipschitz and the restriction
	$\pi|_K$ has displacement $\leq \lambda$.
\end{definition}

\begin{lemma}\label{lem_quasiflat_to_bilipschitzflat}
	Let $X$ be a length space and  let $\Phi:B\to X$ be an  $L$-Lipschitz $(L,A)$-quasidisk. Then there exist $\bar L$ depending only on $L$ and $A$, a metric space $\bar X$, 
	an $L$-bilipschitz embedding $X\to \bar X$ and an $L$-Lipschitz retraction $\pi:\bar X\to X$ with the following additional properties. $\bar X$
	contains a $\bar L$-bilipschitz disk $\varphi:B\to\bar X$ such that $d(\Phi,\varphi)\leq L$ and $d_H(X,\bar X)\leq L$ and the map
	$\Phi$ factors as $\Phi=\pi\circ\varphi$.

\end{lemma}

\begin{proof}
	We glue $B\times[0,L]$ to $X$ along $B \times\{L\}$ via $\Phi$ and denote the resulting space by $\bar X$. 
	We equip  $\bar X$ with the induced length metric \cite[Definition 3.1.12]{BBI_metric} where we view $B\times[0,L]$ 
	as a flat cylinder of width $L$ and identify $X$ with its image in $\bar X$. Then the natural projection $\pi:\bar X\to X$ is L-Lipschitz and
	the canonical embedding $X\hookrightarrow\bar X$ is $L$-bilipschitz.
	We define $\varphi$ as the canonical embedding $B\hookrightarrow B\times\{0\}\subset\bar X$. The distance bound between $\Phi$ and $\varphi$ is clear. 
	It remains to show the bilipschitz property. In the following we will denote by $d_Y$ the metric measured in a length space $Y$ (not to be confused
	with the Hausdorff metric $d_H$.) To simplify notation we will identify $B$ with its image under $\varphi$.
  Let $x,y$ be points in $B$. If $d_{\bar X}(x,y)<2L$, then we have $d_{\bar X}(x,y)=d_{B}(x,y)$ since $B$ is convex in $B\times[0,L]$.
	Moreover, if $d_{B}(x,y)\geq 2LA$, then $d_{\bar X}(x,y)\geq\frac{1}{2L^2}d_{B}(x,y)$ since $\pi$ is $L$-Lipschitz and 
	$\Phi$ is a $(L,A)$-quasi-isometric embedding.
	Finally, for points $x, y\in B$ with $2L<d_{B}(x,y)< 2LA$ holds $d_{\bar X}(x,y)\geq 2L\geq\frac{1}{A}d_{B}(x,y)$ and thus $\varphi$
	is $(2L^2+A)$-bilipschitz as required.
\end{proof}

\begin{corollary}\label{cor_quasiretraction}
	Let $X$ be a length space and  let $\Phi:B\to X$ be an $n$-dimensional  $L$-Lipschitz $(L,A)$-quasidisk with image $D$. Then there exist constants $\lambda_1$ and $\lambda_2$ which depend 
	only on $L,A$ and $n$, and a $\lambda_1$-Lipschitz quasiretraction $\pi:X\to D$ such that $d(x,\pi(x))\le \lambda_2$ for any $x\in D$. The map $\pi$ factors as $\pi=\pi''\circ\pi'$ 
	with a Lipschitz map $\pi':X\to\R^n$. Moreover, $\lambda_2\to 0$ as $A\to 0$.
\end{corollary}

\begin{proof}
	Choose a thickening $\bar X$ as in Lemma \ref{lem_quasiflat_to_bilipschitzflat} and denote by $\bar D\subset \bar X$ the image of the bilipschitz disk $\varphi$ close to $D$.
	By McShane's extension lemma we obtain a Lipschitz retraction $\pi':\bar X\to \bar D$ where the Lipschitz constant is controlled by $L,A,n$. Composing with the natural projection
	$\pi'':\bar X\to X$ we obtain the required map since $D$
	is at distance $\leq L$ from $\bar D$.
\end{proof}

\begin{lem}
	\label{lem_quasiflat_are_quasiminimizers}
Suppose $X$ is a metric space with convex geodesic bicombing and base point $p$. Let $\Phi:B\to X$ be an $n$-dimensional $L$-Lipschitz $(L,A)$-quasidisk with image $D$. 
Then there exist $a,\theta,\Theta$ depending only on $L,A,n$ and $d(p,D)$ 
such that the following holds. There exists an element $T\in\I_{n,loc}(X)$ such that
\begin{itemize}
	\item $\spt(T)\subset D$ and $d_H(\spt(T), D)\leq a$;
	\item (Upper density bound) $\M(T\on B_p (r))\le \Theta\cdot r^n$ whenever $r\ge a$.
	\item (Lower filling bound) Let $S_r=\slc{T,d_p,r}$. Then $\Fill (S_r)\ge \theta\cdot r^n$ whenever $a\leq r\leq d(p,\spt(\D T))$.
\end{itemize}
\end{lem}

\begin{proof}
We take $T=\Phi_\#\bb{B}$ where $\bb{B}\in\I_{n, loc}(\R^n)$ denotes the fundamental class of $B$. Then $\spt(T)\subset\Phi(\spt(\bb{B}))=D$.
By Lemma~\ref{lem_quasiflat_to_bilipschitzflat}, we find a $\bar L$-bilipschitz embedding $\varphi:B\to\bar X$ such that $\Phi=\pi\circ\varphi$ where 
$\pi:\bar X\to X$ is the natural $L$-Lipschitz retraction. Denote by $\psi:\bar X\to \R^n$ a $\tilde L$-Lipschitz extension of $\varphi^{-1}$
provided by McShane.
Then $d(\id_{\R^n},\psi\circ\Phi)\leq\tilde L\cdot L$ and therefore $d_H(\spt(\psi\circ\Phi)_\#\bb{B},B)\leq \tilde L\cdot L$.
Note that  $\spt(\psi\circ\Phi)_\#\bb{B}\subset\psi(\spt T)$. Thus the first item follows since $\Phi\circ\psi$ is uniformly bounded on $D$.

Let $q\in D$ be a point with $d(p,D)=d(p,q)$. Then $\M(T\on B_p(r))\leq\M(T\on B_q(r))\leq L^n(Lr+A)^n$.

The last part follows from \cite[Proposition 3.6]{higherrank} and Lemma~\ref{lem:density}.
\end{proof}

The following lemma is a consequence of \cite[Theorem 2.9]{wenger2005isoperimetric} and \cite[Proposition 2.10]{wenger2005isoperimetric}.

\begin{lemma}
	\label{lem_homotopy}
	Let $X$ be a metric space with a convex geodesic bicombing and let $D\subset X$ an $n$-dimensional $L$-Lipschitz $(L,A)$-quasidisk with $\lambda$-Lipschitz quasiretraction $\pi:X\to D$.
	Then there exists a constant $C>0$ depending only on $L,A$ and $k$ such that the following holds.
	Let $\tau\in \I_k(X)$ and let $h:[0,1]\times \spt(\tau)\to X$ be the geodesic homotopy from $\spt(\tau)$ to $\pi(\spt(\tau))$.
	Suppose that $d(x,\pi(x))\leq \rho$ for all $x\in\spt(\tau)$. Then $h$ induces elements $H\in \I_{k+1}(X)$ and $H'\in \I_{k}(X)$ such that 
	\bit
	\item $\D H=\tau-\pi_\# \tau-H'\quad \text{and}\quad \M(H)\leq C\cdot\lambda^k\cdot\rho\cdot\M(\tau)$;
	\item $\D H'=\D\tau-\pi_\# \D\tau\quad \text{and}\quad \M(H')\leq C\cdot\lambda^{k-1}\cdot\rho\cdot\M(\D\tau)$.
	\eit
	
\end{lemma}

\begin{cor}
	\label{cor_homotopy}
	Let $D,\pi,\lambda_1,\lambda_2$ be as in Corollary~\ref{cor_quasiretraction}. 
	Let $X$ be a metric space with a convex geodesic bicombing. 
	Then there exists a constant $C>0$ depending only on $L,A$ and $k$ such that the following holds.
	Let $\tau\in \I_k(X)$  be such that $\spt(\tau)\subset N_{\rho}(D)$ with $\rho>\lambda_2$.
	Then there exist elements $H\in \I_{k+1}(X)$ and $H'\in \I_{k}(X)$ with $\D H=\tau-\pi_\# \tau-H'$ and $\D H'=\D\tau-\pi_\# \D\tau$. Moreover,
	\[\M(H)\leq C\cdot\lambda_1^k\cdot\rho\cdot\M(\tau).\]
	
\end{cor}

\begin{proof}
	By Lemma~\ref{lem_homotopy}, it suffices to show $d(x,\pi(x))\le C\cdot\rho$ for any $x\in \spt(\tau)$. 
	Let $z$ be a point in $D$ such that $d(x,z)=d(x,D)\le \rho$. Then 
	$d(x,\pi(x))\le d(x,z)+d(z,\pi(z))+d(\pi(z),\pi(x))\le \rho+ \lambda_2+\lambda_1\cdot d(z,x)\le \rho+ \rho+\lambda_1\cdot\rho=(2+\lambda_1)\cdot\rho$.
\end{proof}

\begin{cor}
	\label{cor_homotopy_topdim}
	Let $D,\pi,\lambda_1,\lambda_2$ be as in Corollary~\ref{cor_quasiretraction}. 
	Then there exists a constant $C>0$ depending only on $L,A$ and $n$ such that the following holds.
	Let $\sigma\in \bZ_n(X)$  be such that $\spt(\sigma)\subset N_{\rho}(D)$ with $\rho>\lambda_2$. Then
	\[\Fill(\si)\leq C\cdot\lambda_1^n\cdot \M(\si). \]
	
\end{cor}

\begin{proof}
Corollary~\ref{cor_homotopy} provides a controlled homology $H$ between $\si$ and $\pi_\# \si$.
By Corollary~\ref{cor_quasiretraction}, $\pi$ factors as $\pi=\pi''\circ\pi'$ with a Lipschitz map $\pi':X\to\R^n$.
However, $\pi'_\# \si$ is a top-dimensional cycle
and therefore trivial. It follows that $H$ is a filling of $\si$ as required.
\end{proof}

\section{Informal discussion of the proof of Theorem~\ref{thm_sublinear_close_intro}}~
\label{sec_informal_proof}

As we have mentioned in the introduction, the proof of Theorem~\ref{thm_sublinear_close_intro} follows the standard strategy 
for uniqueness of tangent cones, namely an induction on scales.
However, our implementation of that strategy 	is quite different from previously treated cases. The induction step uses 
properties of Morse quasiflats in an essential way --- the coarse piece decomposition.

Let $Q$ be a Morse quasiflat in a $\cat(0)$ space $X$. Choose a basepoint $p\in X$, and let $C_p (\si)$ denote the geodesic cone based at $p$ 
over a sphere $\si\subset Q$ of very large radius $r_0$. We want to understand the relative position of $C_p (\si)$ and $Q$
on all scales $r<r_0$. Denote by $\si_r$ and $S_r$ the slice of $C_p(\si)$ respectively $Q$ at distance $r$ from $p$. 
Let $T_r=C_p(\si_r)$ and $Q_r=Q\cap B_p(r)$. Note that the occurring objects naturally carry the structure of currents, and in the following discussion
we will use the same notation to denote the respective cycles, currents or sets. (It is actually quite important that $\si_r$ and $S_r$ are cycles, see Remark~\ref{rmk:cycle}.)

In principle, we would like to show that $T_{r}$ and $Q_r$ are Hausdorff close on scale $r$. However, Hausdorff distance does not behave well in our setting
and we are naturally led to use cycles or currents rather than subsets. The reader may wish to think of $Q_0:=Q_{r_0}$ and $T_0:=T_{r_0}$
as relative cycles in $(X,X\setminus B_p(r))$ for every $r<r_0$.

To motivate the setup for the induction argument, we examine the possibilities for the large-scale behavior of our relative cycles.  By Wenger-compactness, at large scales, our configuration is ``close to'' a configuration in an asymptotic cone. Hence we discuss the setting in asymptotic cones first, but
retain to our previous notation. When switching to an asymptotic cone, 
by the definition of Morse quasiflats, $Q$ becomes a bilipschitz flat which is ``homologically injective'', for every $q\in Q$ the inclusion induces injective maps
\[H_n(Q,Q\setminus\{q\})\to H_n(X,X\setminus\{q\}).\]
By \cite[Proposition 6.11]{HKS}, this topological condition translates to a ``piece decomposition'' 
for nearby relative cycles\footnote{Strictly speaking, the transition from homological injectivity to piece decomposition carried out in \cite{HKS}, is not needed here and is only used for illustratory purposes.}.
For any $n$-chain $\al$, 	
whose boundary $\D \al$ is a cycle in $Q$, there is a \emph{piece decomposition} 
\begin{equation}
\al=\ga+\beta
\label{eq:pdec}
\end{equation}

\noindent
where $\ga$ is the canonical filling of $\D\al$ inside $Q$, and $\beta$ is a cycle carried by $X\setminus Q$. 
Moreover, there is ``no cancellation'' between $\ga$ and $\beta$  in the sense that volume --- counted with multiplicities --- behaves additively. 
This can be made rigorous in the language of currents. 
The summands $\ga$ and $\beta$ in such a decomposition are called ``pieces'' of $\al$.
Note that this is indeed a non-trivial property of $Q$.

To further illustrate the piece decomposition, let us consider the case $n=1$. See Figure~\ref{fig:1} below.
In this case, the bilipschitz flat $Q$
is a line $\ell$, and the homological condition simply means that each point $x\in Q$ is a cut point of $X$, separating the two halves of $Q$.
To explain the piece decomposition, let $p$ and $q$ be points on $Q$ and let $\al$ be a path in $X$ from $p$ to $q$.
If $\ga$ denotes the arc in $Q$, with endpoints $p$ and $q$, then clearly $\al$ has to cover $\ga$. More precisely, when counting multiplicities,
$\al$ has to pass through every point on $\ga$ exactly once. So we expect a decomposition $\al=\ga+\beta$, where there is ``no cancellation''
between $\ga$ and $\beta$. Moreover, $\beta$ is a cycle, as whenever $\alpha$ leaves $Q$, it has to come back at exactly the same point. 

\begin{figure}[h!]
	\centering
	\includegraphics[scale=0.8]{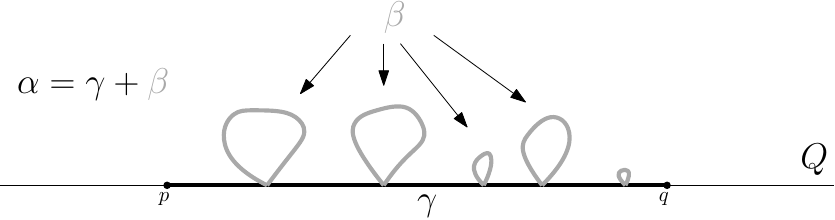}
	\caption{Figure~\ref{fig:1}: Piece decomposition in dimension 1}
	\label{fig:1}
\end{figure}

There is a  relative version of the piece decomposition which states the following. If $\al$ is a relative cycle in $(X,X\setminus B_p(r))$ which 
is relatively homologous to $Q_0$ in $(X,X\setminus B_p(r))$,
then $Q_r$ is a piece of $\al$. Hence for every $r<r_0$ we obtain the piece decomposition
\begin{equation}
T_r=Q_r+\mathcal{B}_r
\label{eq:relpdec}
\end{equation}
where $\mathcal{B}_r$ is a relative cycle in $(X,X\setminus B_p(r))$.
An immediate consequence is a sliced version, namely each cycle $\si_r$ admits a piece decomposition:
\begin{equation}
\si_r=S_r+\beta_r.
\label{eq:slpdec}
\end{equation}
where $\beta_r$ is a cycle carried outside of $Q$.

For a concrete example, consider the space $Y$, obtained from $\R^2$ and $\R^3$ by identifying along a straight line $\ell$. Take a base point $o\in\ell$. The bilipschitz Morse quasiflat $Q$
is given by the $\R^2$-part of $Y$.
To illustrate the piece decomposition in $Y$ (see Figure~\ref{fig:2}), consider a 2-chain $\al$, whose boundary is the fundamental class of the unit circle in the $\R^2$-part.
Then $\al$ can be written as 
$\al=\ga+\beta$, where $\ga$ is the fundamental class of the unit disk in the $\R^2$-part and $\beta$ is a 2-cycle which is entirely contained in the $\R^3$-part.

\begin{figure}[h!]
	\centering
	\includegraphics[scale=0.8]{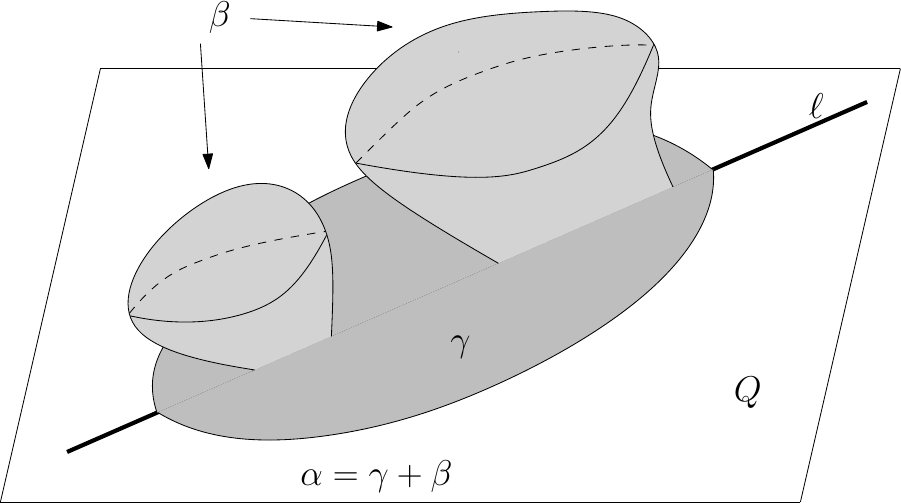}
	\caption{Figure~\ref{fig:2}: Piece decomposition in dimension 2}
	\label{fig:2}
\end{figure}
For the relative/sliced version, let us consider a conical example in $Y$ as follows. See Figure~\ref{fig:3} below. Let $\si$ be the sum of the fundamental class of the unit sphere
in the $\R^2$-part and an arbitrary nontrivial 1-cycle $\si'$ in the unit sphere around $o$ in $Y$. Let $\al$ be the relative 2-cycle obtained by coning off
$\si$ at $o$. Since $\si'$ bounds in the unit sphere of $Y$, $\al$ and $Q$ are relatively homologous in $(Y,Y\setminus B_o(1))$.  
By the (sliced) piece decomposition, $\al$ contains the fundamental class of the $r$-disk in the $\R^2$-part as a piece, which corresponds precisely to $Q_r$;
and the $r$-slices of $\al$ contain $S_r$ as a piece.

\begin{figure}[h!]
	\centering
	\includegraphics[scale=0.8]{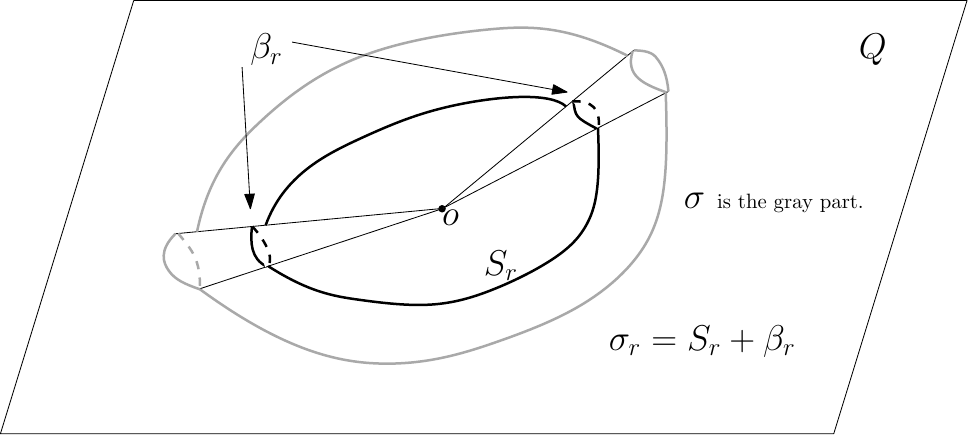}
	\caption{Figure~\ref{fig:3}: relative/slice version of piece decomposition}
	\label{fig:3}
\end{figure}

A naive attempt for the induction on scales argument might be to show that the normalized flat distance between slices stays small. However,
we want to point out here that in our example
the normalized flat distances between the $r$-slices of $\al$ and the fundamental classes of the $r$-spheres in the $\R^2$-part
is constant in $r$ (it equals the normalized filling volume of $\si'$), i.e. it does not decay as $r$ shrinks. 
  In passing from the asymptotic cone to the original space some approximation error is inevitable, and so the borderline monotonicity (of normalized flat distance between slices) in the asymptotic cone can only yield an estimate which allows for some increase when passing from one scale to the next lower scale.  Since there is no a priori bound on the number of scales involved, such an approach will fail. 

The key ingredient for our argument is the following consequence of the piece decomposition.
If $\si_{2r}$ and $S_{2r}$ are homologous in $X\setminus B_p(r)$, then  $S_{r}$ is a piece of $\si_{r}$.
The topological assumption here can be guaranteed by a quantitative bound on flat distance  between
$\si_{2r}$ and $S_{2r}$.

Let us conclude the asymptotic cone discussion and return to the original space.
All assertions that were true in asymptotic cones are now degraded to approximate assertions. 
Moreover, qualitative statements need to be replaced by quantitative ones.  To simplify notation in the following, we will use a hat to indicate normalized quantities. For instance
$\hat M(W)$ for an $n$-current $W$ in the $r$-ball around $p$ refers to $\M(W)/r^n$.
To state the required coarse piece decomposition for large scales $r$, we choose small constants $\eps$ and $C$ with $\eps\ll C$.
If $\al_{2r}$ denotes a current with $\hat\F(\D\al_{2r},S_{2r})<C$, and $\al_r$ denotes its restriction to the $r$-ball around $p$, then there exists a piece decomposition
\begin{equation}
\al_r=\ga_r+\beta_r,\quad \M(\al_r)=\M(\ga_r)+\M(\beta_r), \quad \hat\F(\ga_r,Q_r)<\eps. 
\label{eq:cpdec}
\end{equation}
Note that $\ga_r $ is typically not a relative cycle in $(X,X\setminus B_p(r))$. So $\hat\F(\ga_r,Q_r)<\eps$ means the following. 
There is a chain $\tau_r$ carried in $B_p(r)$ with $\hat\M(\tau_r)\le \eps $ such that $\ga_r+\tau_r$ becomes a relative cycle.
Moreover there is a chain $W_r$ with $\hat\M(W_r)\le \eps$ such that $\D W_r =(\ga_r+\tau_r)-Q_r$ modulo a chain in $X\setminus B_p(r)$. 
Similarly, $\beta_r$ differs from a relative cycle by $\tau_r$.

Again, we obtain a corresponding version for slices.
If the normalized flat distance between $\si_{2r}$ and $S_{2r}$ is at most $C$, then 
\begin{equation}
\si_r=\mu_r+\nu_r, \quad \M(S_r)=\M(\mu_r)+\M(\nu_r), \quad \hat\F(\mu_r,S_r)<\eps.
\label{eq:scpdec}
\end{equation}

This suggests a potential induction scheme:   
If  $S_{r_{k}}$ is $C$-close to a piece of $\si_{r_{k}}$, then 
$S_{r_{k+1}}$ is $C$-close to a piece of $\si_{r_{k+1}}$ where $r_k=\frac{r_0}{2^k}$.
This is, in a somewhat oversimplified form, what we are proving.

At this point one might get the impression that the (sliced) coarse piece decomposition directly provides a proof
for the induction step.
However, in addition to technical difficulties, the implementation of this strategy has to overcome 
a conceptional issue, which stems from the incompatibility of the coarse piece decomposition and the conical structure.

More precisely, since the piece $\mu_{r}$ is not a cycle, the flat distance estimate $\hat{F}(\mu_r,S_r)<\eps$
means that there exists a current $\mathcal{E}_r$ --- thought of as an error term --- such that 
$\mu_r+\mathcal{E}_r$ is a cycle, the flat distance estimate $\hat{F}(\mu_r+\mathcal{E}_r,S_r)<\eps$ holds, and $\mathcal{E}_r$ is small,
$\hat\M(\mathcal{E}_r)<\eps$. To understand the trouble here, we remind the reader that our ultimate goal is to show that our quasiflat lies close to
a geodesic cone. While the error term is small, it does not lie in the cone anymore.
Now let us examine the information provided by the coarse piece decomposition for the induction step. 
Instead of the cycle $\si_{r_{k+1}}$, we consider an auxiliary cycle $\al_{r_{k+1}}$, the radial projection to scale $r_{k+1}$
of the cycle $\mu_{r_{k}}+\mathcal{E}_{r_{k}}$. We know that, up to a small error, $\al_{r_{k+1}}$ contains a good piece $\ga_{r_{k+1}}$
which lies close to  $S_{r_{k+1}}$, and a bad piece $\beta_{r_{k+1}}$. Unfortunately, the good piece might contain the projection of the error $\mathcal{E}_{r_{k}}$ from the last scale. Therefore,  an application of the coarse piece decomposition increments the size of the total error for $\si_{r_{k+1}}$ by $\eps$.
From the overall perspective, this means that we deviate from the geodesic cone even further.
While a single such error is small by choice, we definitely must prevent an accumulation of these errors.

We resolve this issue as follows. First, if the flat distance estimate  $\hat\F(\si_{r_{k+1}},S_{r_{k+1}})<C$ happens to hold, we do not have to rely on the coarse piece decomposition at all, and therefore we do not have to introduce an error term. So let us assume
the flat distance estimate fails,  $\hat\F(\al_{r_{k+1}},S_{r_{k+1}})\geq C$, and let's bring in 
the coarse piece decomposition. Then, up to a small error, $\hat\F(\si_{r_{k+1}},S_{r_{k+1}})$ is equal to the filling volume of the bad piece
$\beta_{r_{k+1}}$. By the isoperimetric inequality, the normalized mass of $\beta_{r_{k+1}}$ has to have a definite size, say $\delta$.
This leads to a decay in normalized mass of the occurring good pieces. Namely $\hat\M(\ga_{r_{k+1}})\leq\hat\M(\si)-N\cdot\delta$
where $N$ is the number of previous scales where the flat distance estimate failed and we had to rely on the coarse piece decomposition.
Since the good piece has to be non-trivial, we obtain a uniform bound on the number of such scales.
Hence error terms do not accumulate and our strategy succeeds.

\bigskip

\begin{remark} 
	In the special case when $Q$ is top-dimensional, the above discussion simplifies considerably. 
	Considering asymptotic cone first, the piece decomposition becomes trivial as the $\beta$-term in \eqref{eq:pdec} disappears, hence \eqref{eq:relpdec} reduces to $T_r=Q_r$ and the respective sliced version \eqref{eq:slpdec} reduces to $\si_r=S_r$. Inside the space one no longer need the coarse piece decomposition. As in \cite{higherrank}, one can argue directly $\hat\F(\al_r,Q_r)<\eps$ through induction on scales, which leads to a much simpler argument.
\end{remark}

\section{Morse quasiflats and coarse decompositions}
\label{sec_definition}
In this section we recall the notion of Morse quasiflats from \cite{HKS}, as well as the coarse neck property and coarse piece property of Morse quasiflats.

Let $X$ be a complete metric space with a convex geodesic bicombing and let $F\subset X$ be the image of an $L$-bilipschitz
embedding of a closed convex subset of $\R^n$. 
We begin with a purely topological condition,\ cf. \cite[Definition~6.8]{HKS}.

\begin{definition}[Full support]
	\label{def_full_support}
 We say $F$ has {\em full support}, if the map
 \[ H_n(F,F\setminus\{q\},\mathbb Z)\to  H_n(X,X\setminus\{q\},\mathbb Z)\] 
 on reduced singular homology is injective for each $q\in F\setminus\D F$. 
\end{definition}

\begin{definition}[Morse quasiflat]
	\label{def_morse}
 An $n$-dimensional quasiflat $Q\subset X$ is called {\em Morse}, if for any asymptotic cone $X_\om$ of $X$ the ultralimit 
$Q_\om\subset X_\om$ of $Q$ has full support in $X_\om$. 

Let $\mathcal{C}$ be a collection of $n$-dimensional quasidisks or quasiflats with uniform quasi-isometric constants. $\mathcal{C}$ is \emph{uniformly Morse} if for any limit $D_\om$ of elements from $\mathcal{C}$ in the asymptotic cone $X_\om$ has full support.
\end{definition}

We now review several characterizations of Morse quasiflats. We start with the coarse neck property. To get an intuition, we refer the reader to the informal discussion in Section~\ref{sec_informal_proof}. In particular, in the decomposition $\alpha=\gamma+\beta$ of Figure~\ref{fig:1} and Figure~\ref{fig:2}, the places where the two pieces $\gamma$ and $\beta$ touch can be thought of as ``necks'' of $\al$. The coarse version of ``necks'' in the space rather than in the asymptotic cone, is characterized by the following definition.

\begin{definition}
	\label{def_tcp}
	An $n$-dimensional quasiflats $Q$ has the \emph{coarse neck property} (CNP), if there exists a constant $C_0>0$, and
for any point $p\in X$ and  given positive constants $\rho$ and $C$, 
there exists $R_{\text{CNP}}=R_{\text{CNP}}(p,\rho,C)$ such that for any $R_{\text{CNP}}\leq R$ the following holds. 

Let $\tau\in \I_n(B_p(C R)\setminus N_{\rho R}(Q))$ with $\si:=\D \tau$ be such that
\[\M(\tau)\le C\cdot R^n,\quad \M(\si)\le C\cdot R^{n-1}\quad \text{and}\quad \spt(\si)\subset N_{2\rho R}(Q).\]

Then we have 
\[\Fill(\si)\le C_0\cdot\rho R\cdot\M(\si).\] 
Note that the definition of CNP depends on the parameter $C_0$ and the function $R_{\text{CNP}}$.	

Similarly we define the \emph{coarse neck property} for a local integral current $T\in \I_{n,loc}(X)$ by replacing $Q$ with $\spt(T)$, and replacing $R_{\text{CNP}}\le R$ by $R_{\text{CNP}}\leq R\leq\frac{d(p,\spt(\D T))}{C}$ (so that $B_p(CR)$ is disjoint from $\D T$).
\end{definition}

The following is a consequence of \cite[Theorem 9.10]{HKS}.

\begin{proposition}
	\label{cor_small_neck}
	Let $X$ be a proper convex geodesic metric space.
	Let $\mathcal{D}$ be a family of quasidisk or quasiflats in $X$ with uniform quasi-isometric constants. Then $\mathcal{D}$ is uniformly Morse if and only each element of $\mathcal{D}$ satisfies the coarse neck property for some uniform $C_0$ and $R_{\text{CNP}}$.
\end{proposition}

In the situation of this proposition, the pair $(C_0, R_{\text{CNP}})$
will be referred to as the {\em Morse data} for $\mathcal{D}$.

The above summarizes what we need from \cite{HKS} in this paper. More discussion on alternative characterizations of Morse quasiflats, as well as connections to the literature on Morse quasigeodesics, can be found in the introduction of \cite{HKS}.

Now we deduce the ``coarse piece decomposition'' mentioned in Section~\ref{sec_informal_proof} from the coarse neck property (Lemma~\ref{lem_decomposition_lemma}). The following is  a version of \cite[Lemma 9.2]{HKS} for quasidisks.

\begin{lem}
	\label{lem_decomposition_lemma}
	Let $D\subset X$ be an $n$-dimensional $L$-Lipschitz $(L,A)$-quasidisk with CNP (cf. Definition~\ref{def_tcp}). 
	Let $p\in X$ be a base point. 
	
	For given $\eps,C>0$, there exists $R_{\text{neck}}>0$ depending only on $\eps, C, d(p,D),$ $L, A, n, X$ and the CNP parameter of $D$ such that the following holds for any $R_{\text{neck}}\le R\le\frac{d(p,\D D)}{C}$. 
	
	Let $T\in \I_{n,c}(X)$ with $\spt(\D T)\subset N_C(D)$, $\spt(T)\subset B_p(CR)$ and $\M(T)\le C\cdot R^n$. 
	Suppose there exists $T'\in \I_{n,c}(X)$ such that $\D T=\D T'$ with $\spt(T')\subset N_C(D)$ and $\M(T')\le C\cdot R^n$. Then $T$ admits a {\em coarse piece decomposition}: a piece decomposition $T=U+V$, 
	induced by the distance function $d_D$, with additional properties. Set $\si=:\partial U-\partial T=-\partial V$ and let $\omega$ be a minimal filling of $\si$. Then
	\begin{enumerate}
		\item $\Fill(\si)\le \eps\cdot R^{n}$;
		\item $\Fill(U+\omega-T')\le  \eps\cdot R^{n+1}$;
		\item $\spt(U+\omega-T')\subset N_{\eps R}(D)$. 
		\end{enumerate}
\end{lem}

\begin{proof}
	Take a small constant $h$ and a large natural number $K$ whose values will be determined later. Let $T_{x,y}=T\on\{xR\le d_D<yR\}$. 
	By the pigeonhole principle, there exists a natural number  $ k\le K$ such that 
	$\widehat T:=T_{\frac{h}{2^k},\frac{h}{2^{k-1}}}$ satisfies 
	$\M(\widehat T)\le\frac{\cdot \M(T)}{K}$. Hence there exists $r\in(\frac{hR}{2^k},\frac{hR}{2^{k-1}})$ such that 
	$\si:=\slc{T,d_D,r}\in\I_{n-1,c}(X)$ satisfies 
	\[\M(\si)\le \frac{\M(\widehat T)}{hR/2^k}\le \frac{2^k}{KhR}\cdot\M(T)\le \frac{2^{K} C}{Kh}\cdot R^{n-1}.\] 
	Set $U:= T\on\{d_D\le r\}$ and $V=T\on\{d_D> r\}$. 
	
	Then
	\bit
	\item $\spt(\si)\subset N_{2\rho R}(D)$ and $\spt( V)\subset B_p(CR)\setminus N_{\rho R}(D)$ with $\rho:=\frac{h}{2^k}$.
	\item $\M(V)\leq C\cdot R^n$.
	\eit

	Now we apply the coarse neck property to $V$. 
	For $C'=\max\{C,\frac{2^{K} C}{Kh}\}$ and $R\geq \ul{R}(p,\rho,C')$ holds 
	\begin{equation}
	\label{eq_mass_of_tau}
	\Fill(\si)\le C_0\cdot \rho R\cdot\M(\si)\le C_0\cdot\frac{hR}{2^{k}}\frac{2^k }{KhR}\cdot\M(T)=\frac{C_0}{K}\cdot\M(T).
	\end{equation}
	
	Let $\omega$ be a minimal filling of $\si$ and consider the cycle $S=U+\omega-T'$. 
	From the triangle inequality and the minimality of $\omega$, we see $\M(S)\le \M(T)+\M(T')$. We may assume  $C\leq hR$ and therefore $\spt(T')\subset N_{hR}(D)$. 
	By \eqref{eq_mass_of_tau}, we then conclude $\spt(S)\subset N_{(\delta+h)R}(D)$ where 
	$\delta=\left(\frac{2C_0}{K\delta_0} \right)^{\frac{1}{n}}$. By Corollary~\ref{cor_homotopy_topdim}
	\begin{align}
	\label{eq_small_filling}
	\Fill(S)\le  \tilde C(\delta+h)R(\M(T)+\M(T'))
	\end{align}
The lemma follows from \eqref{eq_mass_of_tau} and \eqref{eq_small_filling} by choosing $h$ small and $K$ and $R$ large.
\end{proof}

	\begin{corollary}\label{cor_good_part_covers}
	 Let $D\subset X$ be an $n$-dimensional $L$-Lipschitz $(L,A)$-quasidisk with CNP (cf. Definition~\ref{def_tcp}).
	Let $p\in X$ be a base point. 
	
	For given $\eps,C>0$, there exists $\ul{R}$ depending only on $\eps, C, d(p,D), L, A,$ $n, X$ and the CNP parameter of $D$ such that the following holds for any $\ul{R}\leq R\leq \frac{d(p,\D D)}{C}$. 
	
	Let $T\in \I_{n,loc}(X)$ represent $D$ as in Lemma~\ref{lem_quasiflat_are_quasiminimizers}. 
	Let $\tau\in \I_{n,c}(X)$ be a chain with $\spt(\tau)\subset B_p(CR)$ and $\M(\tau)\le C\cdot R^n$.
	Suppose $\D \tau=\D \tau'$ for a piece $\tau'\in\I_{n,c}(X)$ of $T$ with $\M(\tau')\le C\cdot R^n$. 
	Let $\tau=U+V$ be a coarse piece decomposition as in Lemma~\ref{lem_decomposition_lemma}. 
	Then
	\[\spt(\tau')\subset N_{\eps R}(\spt U).\]
	
	\end{corollary}
	
	\begin{proof}
	 Set $\tilde \eps=\min\{\theta_1(\frac{\eps}{2})^{n+1},\delta_0(\frac{\eps}{2})^{n}\}$ and choose $R_{\text{neck}}$ for given $\tilde \eps, C$ as in  Lemma~\ref{lem_decomposition_lemma}.
	 Let $\omega\in \I_{n,c}(X)$ be a minimal filling of $\si:=-\D V$ and let $W\in\I_{n+1,c}(X)$ be a minimal filling of $U+\omega-\tau'$. Then $\M(\omega)\leq\tilde\eps\cdot R^n$
	and $\M(W)\leq\tilde\eps\cdot R^{n+1}$ by Lemma~\ref{lem_decomposition_lemma}. By the lower densitiy bound for minimizers, we have $\M(\omega)\geq \delta_0\cdot d(x,\spt(\si))^n$ for $x\in\spt(\omega)$, 
	and from Lemma~\ref{lem:fill-density} we obtain $\M(W)\geq \theta_1\cdot d(y,\spt(U+\omega))^{n+1}$ for $y\in\spt(W)$. By our choice of $\tilde\eps$ we arrive at $\spt(\tau')\subset N_{\frac{\eps}{2}R}(\spt(U+\omega))$
	and $\spt(\omega)\subset N_{\frac{\eps}{2}R}(\spt(\si))$. This concludes the proof since $\spt(\si)\subset\spt(U)$.
	\end{proof}

\section{Visibility for Morse quasiflats}
\label{sec_visibility_morse_quasiflats}
This section concerns the large scale structure of Morse quasiflats. Under the assumption of convex geodesic bicombings we prove that Morse quasiflats are asymptotically conical.
We then turn to boundaries at infinity and show that Morse quasiflats have a well defined Tits boundray which itself is a Morse cycle. At last we prove a rigidity result for Morse quasiflats of Euclidean 
mass growth.

\subsection{Asymptotic conicality}
Let $\Phi:\R^n\to X$ be a quasiflat in a metric space $X$ with image $Q\subset X$.
We say that a local current $T\in\bZ_{n,loc}(X)$ {\em represents} $Q$, if $T$ satisfies the following for some $\Lambda, C, a>0$.
\begin{itemize}
	\item $T$ is $(\Lambda,a)$-quasi-minimizing;
	\item $T$ is $(C,a)$-controlled, i.e. $\|T\|(B_p(r))\leq C\cdot r^n$ for all $p\in X$ and $r\geq a$;
	\item $d_H(Q,\spt(T))\le a$.
\end{itemize}
By \cite[Proposition 3.6]{higherrank}, every quasiflat in a proper metric space  with convex geodesic bicombing can be represented by a local current.
Moreover, if there exists a Lipschitz map $\Phi':\R^n\to X$ at finite distance from $\Phi$, then  
$T=\Phi'_\#\bb{\R^n}\in\bZ_{n,loc}(X)$ represents $Q$, cf. Lemma~\ref{lem_quasiflat_are_quasiminimizers}.

\begin{thm}[asymptotic conicality]
	\label{thm_cone_is_close_to_quasiflat}
	Let $X$ be a proper metric space  with convex geodesic bicombing.
Suppose $Q\subset X$ is a Morse  $(L,A)$-quasiflat,  represented by a current $T\in\bZ_{n,loc}(X)$. Let $p\in X$ be a base point. 
Then for any given $\epsilon>0$, there exists $R_{\text{ac}}$ depending only on $\epsilon, L, A, d(p,Q),X$ and the Morse data of $Q$ such that for every $R_{\text{ac}}\leq r$ holds
\[\ccone_p(\spt(T))\cap B_p(r)\subset N_{\eps r}(\spt T).\]
\end{thm}

The proof of \cite[Theorem 8.6]{higherrank} shows that asymptotic conicality is a consequence of the following Theorem~\ref{thm_visibility}, as the proof of \cite[Theorem 8.6]{higherrank} is a packing argument which does not depend on the rank assumption in \cite{higherrank}.

\begin{thm}[visibility property]
	\label{thm_visibility}
	Let $X$ be a proper metric space  with convex geodesic bicombing.
Suppose $Q\subset X$ is a Morse  $(L,A)$-quasiflat,  represented by a current $T\in\bZ_{n,loc}(X)$. Let $p\in X$ be a base point and denote by $S_r$ the slices $\slc{T,d_p,r}$. 
	Then for given $\eps>0$, there exists $R_{\text{vis}}$ depending only on  $\epsilon, L, A, d(p,Q),X$ and the Morse data of $Q$ such that for every $R_{\text{vis}}\leq r\leq R$ holds
	
	\[\spt(T)\cap B_p(r)\subset N_{\eps r}(C_p(\spt(S_{R}))).\]
	
\end{thm}

In the set of this subsection we prove Theorem~\ref{thm_visibility}. Instead of estimating the Hausdorff distance directly, we first establish an estimate for certain filling distance (cf. Lemma~\ref{lem_good_slices}), and then deduce the desired distance estimate (cf. Corollary~\ref{cor_quasiflat_is_close_to_cone}). A key estimate needed in Lemma~\ref{lem_good_slices} is Lemma~\ref{lem_dichotomy}.

Throughout this subsection, we fix a proper metric space $X$ with convex geodesic bicombing and a base point $p\in X$. 
We also fix $Q\subset X$, an $n$-dimensional Morse $(L,A)$-quasiflat.
Recall that $\theta$ is the lower filling bound for spherical slices of $T$ (cf. Lemma~\ref{lem_quasiflat_are_quasiminimizers}) and $\Theta$ is the upper density of bound of $T$.

Let $T$ be a current representing $Q$ as in Lemma~\ref{lem_quasiflat_are_quasiminimizers}. We denote the slice $\slc{T,d_p,r}$ by $S_r$. It is called {\em generic}, if 
$S_r\in\I_{n-1,c}(X)$ and $\M(S_r)\leq\hat\Theta\cdot r^{n-1}$ with $\hat\Theta=4\cdot\Theta$. By the coarea inequality, for every $r_0>0$ there exits 
a generic slice in the range $[\frac{r_0}{4},\frac{r_0}{2}]$.

\begin{lem}\label{lem_lower_fill_bound}
Let $c\leq \frac{\theta}{2}$ be given. Suppose that $\al_r\in\I_{n-1,c}(B_p(r))$ is a cycle with $\F(\al_r,S_r)\leq c \cdot r^n$ for some $r\geq a$.
 Then $\M(\al_r)\geq c \cdot r^{n-1}$. 
\end{lem}

\begin{proof}
 By the lower filling bound of Lemma~\ref{lem_quasiflat_are_quasiminimizers} and our assumption, we have $\Fill(\al_r)\geq c\cdot r^n$. The claim follows from the coning inequality. 
\end{proof}

Let $\de_0$ be as in Theorem~\ref{thm:plateau}. The following is a consequence of Theorem~\ref{thm:plateau}.

\begin{lem}\label{lem_spt_outside_ball}
Suppose that $\al_r\in\bZ_{n-1,c}(X)$ is a cycle with $\spt(\al_r)\subset S_p(r)$.
Then there exists a positive constant $c_1=c_1(\delta_0,r_0)$ such that  $\F(\al_r,S_r)\leq c_1\cdot r^n$ implies that any minimal filling of $\al_r-S_r$ is supported
in $X\setminus B_p(\frac{r}{2})$ for all $r\geq r_0$.
\end{lem}

To be in the range of both preceding lemmas we set $$c_2:=\min\{\frac{\theta}{2},c_1(\delta_0,1)\}.$$

\begin{lem}\label{lem_dichotomy}
Let $0<c\leq c_2$ be given.
 There exists a small positive constant $\vartheta=\vartheta(c)$, such that for every $\delta>0$ there exists a large radius $R$, 
 depending on $\delta, d(p,Q), \Theta, c,X, L, A, n$ 
 and the CNP parameter of $Q$ such that the following holds for all $r_1\geq 4\cdot R$.
 
 Let $\al_1\in\bZ_{n-1,c}(X)$ be a cycle supported in $S_p(r_1)$ and such that 
 \bit
 \item $\M(\al_1)\leq\hat\Theta\cdot r_1^{n-1}$;
 \item $\F(\al_1,S_{r_1})\leq c\cdot r_1^n$.
 \eit
 Set $\al_r=h_{p,\frac{r}{r_1}\#}\al_1$. Suppose $\F(\al_r,S_{r})>c\cdot r^n$ for almost all $r\in[\frac{r_1}{4},\frac{r_1}{2}]$.
 Then there exists a generic slice $S_{r_2}$ with $r_2\in[\frac{r_1}{4},\frac{r_1}{2}]$ and a cycle $\al_2\in\bZ_{n-1,c}(X)$ supported in $S_p(r_2)$ such that 
 \bit
 \item $\frac{\M(\al_2)}{r_2^{n-1}}\leq \frac{\M(\al_1)}{r_1^{n-1}}-\vartheta$\quad (mass drop);
 \item $\F(\al_2,S_{r_2})\leq \delta\cdot r_2^n$\quad (closeness).
 \eit
 Moreover, $\al_2=\gamma_2+\beta_2$  where the {\em good part} $\gamma_2$ is supported in $C_p (\spt(\al_1))$ and
 the {\em bad part} $\beta_2$ is small, $\M(\beta_2)\leq \delta\cdot r_2^{n-1}$.
\end{lem}

\begin{proof}
 Set $\vartheta=\frac{c}{2}$.
 We will choose $R$ large enough such that Lemma~\ref{lem_decomposition_lemma} applies for appropriate choices of $\eps$ and $C$. 
 We set $C=\hat\Theta+c$ and choose $\eps<\min\{\frac{c}{2^{n+5}}, \frac{\delta}{2^{n+3}}\}$.
 Now we choose $R=\ul R(\eps,C)$ where $\ul R$ is provided by   Lemma~\ref{lem_decomposition_lemma}.

 Let $\tau\in\I_{n,c}(X\setminus B_p(\frac{r_1}{2}))$ be a filling of $\al_1-S_{r_1}$ with $\M(\tau)\leq c\cdot r_1^n$ (cf. Lemma \ref{lem_spt_outside_ball}). 
 We apply Lemma~\ref{lem_decomposition_lemma} to the chain $\cone_p(\al_1)+\tau$ to produce a coarse piece decomposition
$\cone_p(\al_1)+\tau=U+V$ where $U$ denotes the piece supported close to $Q$. Recall that Lemma~\ref{lem_decomposition_lemma} provides a filling $\omega\in \I_{n,c}(X)$ of $\D V$ with mass $<\eps \cdot r_1^n$
and a filling $W$ of $U+\omega-(T\on B_p(r_1))$ with mass $<\eps\cdot r_1^{n+1}$.

We will now slice $U,V,W$ and $\omega$ in the range $[\frac{r_1}{4},\frac{r_1}{2}]$  with respect to $d(\cdot,p)$ to obtain controlled slices  $U_r$, $V_r$, $W_r$ and $\om_r$. 
Note that since $\spt(\tau)\subset S_p(r_1)$, we have the piece decomposition $\al_r=U_r+V_r$; we know that $U_r+\omega_r$ and $V_r-\omega_r$ are cycles; and
we see that $W_r$ is a filling of $U_r+\omega_r-S_r$. 

We define $\al_2:=U_r+\omega_r$, where $U_r$ corresponds to the good part $\ga_2$ and $\omega_r$ corresponds to the bad part $\beta_2$ in the statement of the lemma.

By the coarea inequality we can choose $r\in[\frac{r_1}{4},\frac{r_1}{2}]$ such that 
\[\M(\omega_{r})\leq 2^{n+2}\cdot\eps\cdot r^{n-1} \text{ and } \M(W_{r})\leq 2^{n+3}\cdot\eps\cdot r^{n}.\]
By our choice of $\eps$, this implies $\M(W_{r})\leq \delta\cdot r^n$
and we conclude $\F(\al_2,S_{r})\leq\M(W_{r})\leq\delta \cdot r^n$ as required.

Note that $U_{r}$ is supported in $C_p (\spt(\al_1))$. Again, by our choice of $\eps$, we see $\M(\omega_{r})\leq\delta\cdot r^{n-1}$. 
Hence $\al_2$ has the desired good/bad decomposition.

Using the triangle inequality and the coning inequality, we estimate
\begin{align*}
 \F(\al_r,S_r)&\leq \F(\al_r,\al_2)+\F(\al_2,S_r)\\
 &= \Fill(V_r-\omega_r)+\F(U_r+\omega_r,S_r)\\
 &\leq r\cdot \M(V_r-\omega_r)+\M(W_r).
\end{align*}
By assumption $\F(\al_r,S_r)>c\cdot r^n$ and we conclude 
\[\M(V_r-\omega_r)\geq (c-2^{n+3}\eps)\cdot r^{n-1}.\]
Hence 
\begin{align*}
\M(\al_2)&\leq \M(U_r)+\M(\omega_r)=\M(\al_r)-\M(V_r)+\M(\omega_r)\\
&\leq \frac{\M(\al_1)}{r_1^{n-1}}\cdot r^{n-1}-\M(V_r-\omega_r)+2\M(\omega_r)<(\frac{\M(\al_1)}{r_1^{n-1}}-\vartheta) \cdot r^{n-1}.
\end{align*}
The second step uses the piece decomposition $\al_r=U_r+V_r$. We complete the proof by choosing $r_2=r$.
\end{proof}

\begin{lem}\label{lem_good_slices}
 Let $c$ and $\delta$ be given such that $0<\delta\leq c\leq c_2$. Then there exists a large radius $R$, depending on $\delta, d(p,Q), \Theta, c, X, L, A, n$ 
 and the CNP parameter of $Q$ such that the following holds. 
 
 For any generic cycle $S_{r_0}$ with $r_0>R$ we find a
  cycle $\al_r\in \bZ_{n-1,c}(X)$ with $\spt(\al_r)\subset S_p(r)$ and $r\in[R,4R]$ such that
 
 \begin{enumerate}
		\item $\F(\al_r, S_r)\leq c\cdot r^n$; in particular, any minimal filling of $\al_r-S_r$ is carried by $X\setminus B_p(\frac{r}{2})$;
		\item $\al_r$ can be written as a sum of a {\em good part} and a {\em bad part}, $\al_r=\ga_r+\beta_r$;
		\item the good part $\ga_r$ is nontrivial and satisfies $0<\M(\ga_r)\le \hat\Theta\cdot r^{n-1}$ and $\spt(\ga_r)\subset\cone_p(\spt(S_{r_0}))$;
		\item the bad part $\beta_r$ is small, $\M(\beta_r)\le \delta\cdot r^{n-1}$.
	\end{enumerate}
\end{lem}

\begin{proof}

Set $\delta'=\frac{\delta\vartheta}{\hat\Theta-c}$ where $\vartheta=\frac{c}{2}$ as before and choose $R=R(\delta')$ as in Lemma~\ref{lem_dichotomy}.

 We inductively define a sequence of cycles $\al_k$ supported in $S_p(r_k)$ with $r_k\in[\frac{r_0}{4^k},\frac{r_0}{2^k}]$
 such that each cycle $\al_k$ has the required properties on its own scale.
 
 We set $\al_0=S_{r_0}$. To define $\al_{k+1}$ we distinguish two cases. 

 \noindent
{\em Case 1.} There exists $r\in [\frac{r_0}{4^k},\frac{r_0}{2^k}]$  such that
 $\F(h_{p,\frac{r}{r_k}\#}(\al_k),S_r)\leq c\cdot r^n$. 

In this case we set $\al_{k+1}=h_{p,\frac{r}{r_k}\#}(\al_k)$. (See 
Section~\ref{subsec:bicombing} for the definition of $h_p$.)  The good/bad decomposition of $\al_{k}$ induces 
a good/bad decomposition of $\al_{k+1}$, $\ga_{k+1}=h_{p,\frac{r_{k+1}}{r_k}\#}(\ga_k)$ and $\beta_{k+1}=h_{p,\frac{r_{k+1}}{r_k}\#}(\beta_k)$.

\medskip
\noindent
{\em Case 2. Negation of Case 1.}
 Now we apply Lemma~\ref{lem_dichotomy} to obtain $\al_{k+1}$ from $\al_{k}$.
 Lemma~\ref{lem_dichotomy} also provides a good/bad decomposition  $\al_{k+1}=\ga'_{k+1}+\beta'_{k+1}$ where $\spt(\ga'_{k+1})\subset \cone_p(\spt(\al_{k}))$ and
 $\M(\beta'_{k+1})\leq \delta'\cdot r_{k+1}^{n-1}$. We define $\ga_{k+1}=\al_{k+1}\on\cone_p(\spt(\al_0))$ and $\beta_{k+1}=\al_{k+1}-\ga_{k+1}$.
 Hence $\M(\beta_{k+1})\leq N\cdot\delta'\cdot r_{k+1}^{n-1}$ where $N$ is the number of times Case 2 previously occured.
 We claim that $N$ is uniformly bounded. Indeed, by Lemma~\ref{lem_lower_fill_bound}, we know that 
 $\M(\al_{k+1})\geq c \cdot r_{k+1}^{n-1}$. On the other hand, by Lemma~\ref{lem_dichotomy},
 $\M(\al_{k+1})\leq (\hat\Theta-N\vartheta)\cdot r_{k+1}^{n-1}$ which provides the upper bound  $N\leq \frac{\hat \Theta-c}{\vartheta}$.
 By our choice of $\delta'$ we get $N\cdot\delta'\leq\delta$ and therefore the required mass bound for the bad part $\beta_{k+1}$.
 
 Finally we set $\al_r=\al_{l}$ where $l$ is maximal such that $r_l\geq R$. This concludes the proof.
\end{proof}

\begin{cor}
	\label{cor_quasiflat_is_close_to_cone}
	For given  $\eps>0$, there exists $R_0$ depending only on $\eps, L, A, d(p,Q),X$ and the Morse data of $Q$ such that the following holds for any $r\ge R_0$. 
	If $r_0\geq 4r$ and $S_{r_0}$ is a generic slice, then
	\[\spt(T)\cap B_p(r)\subset N_{\eps r}(C_p(\spt(S_{r_0}))).\]
	
\end{cor}

\begin{proof}
We choose $\delta$ and $c$ small enough, such that $\delta<(\frac{\eps}{2})^n\cdot\delta_0$ and $(c+c_0\cdot (2\delta)^{\frac{n}{n-1}})\leq c_1(\delta_0,1)$ where $c_1$ is the constant from Lemma~\ref{lem_spt_outside_ball}.
Then we choose $R(c,\delta)$ as in Lemma~\ref{lem_good_slices}. Set $$C=\hat\Theta+\delta+c+c_0\cdot (2\delta)^{\frac{n}{n-1}}$$ and choose $\ul{R}(\frac{\eps}{2},C)$ as in Corollary~\ref{cor_good_part_covers}.
Finally, set $$R_0=\max\{R(c,\delta), \ul{R}(\frac{\eps}{2},C)\}.$$
 
 By Lemma~\ref{lem_good_slices}, we find a cycle $\al_r\in \bZ_{n-1,c}(X)$ supported in $S_p(r)$ with $r\in[R,4R]$. It decomposes as $\al_r=\ga_r+\beta_r$  such that
 $\F(\al_r, S_r)\leq c\cdot r^n$, $\spt(\ga_r)\subset\cone_p(\spt(S_{r_0}))$ and $\M(\beta_r)\le \delta\cdot r^{n-1}$. 
 
 Let us choose a minimal filling $\omega_r$ of $\D \beta_r$. Then $\M(\omega_r)\leq\delta\cdot r^{n-1}$ by minimality and $\F(\beta_r,\omega_r)\leq c_0\cdot(2\delta)^{\frac{n}{n-1}}\cdot r^n$ by the isoperimetric inequality.
 Since $\spt(\omega_r)\subset N_{\delta_1 r}(\spt(\D \beta_r))$ with $\delta_1=(\frac{\delta}{\delta_0})^{\frac{1}{n}}<\frac{\eps}{2}$, we see 
 $C_p(\spt(\gamma_r+\omega_r))\cap B_p(r)\subset N_{\frac{\eps}{2} r}(C_p(\spt(\gamma_r)))$ by convexity.
 So it is enough to show $\spt(T\on B_p(r))\subset N_{\frac{\eps}{2} r}(C_p(\spt(\gamma_r+\omega_r)))$.
 
 From the triangle inequality we obtain $\F(S_r,\gamma_r+\omega_r)\leq(c+c_0\cdot(2\delta)^{\frac{n}{n-1}})\cdot r^n$. By our choice of $c$ and $\delta$, any minimal filling $\tau_r$ of $S_r-\gamma_r-\omega_r$
 will be carried in $X\setminus B_p(\frac{r}{2})$. We consider the chain $C_p(\gamma_r+\omega_r)+\tau_r$ with boundary $S_r$. It comes with mass control 
 \[\M(C_p(\gamma_r+\omega_r)+\tau_r)\leq (\hat\Theta+\delta+c+c_0\cdot(2\delta)^{\frac{n}{n-1}})\cdot r^n\leq C\cdot r^n.\] 
 Hence Corollary~\ref{cor_good_part_covers} implies 
 \[\spt(T\on B_p(r))\subset N_{\frac{\eps}{2} r}(C_p(\spt(\gamma_r+\omega_r))).\]
\end{proof}

\subsection{The Tits boundary of a Morse quasiflat}
We refer to Section~\ref{subsec:bicombing} for the definition of Tits cone and Tits boundary for a metric space with convex geodesic bicombing.
\begin{definition}
	\label{def:pointed}
A quasiflat $Q\subset X$ is \emph{pointed Morse} if for any asymptotic cone $X_\omega$ of $X$ with fixed base point, the inclusion $Q_\omega\to X_\omega$ induces injective maps on 
 local homology at each point in $Q_\omega$.
\end{definition}

If we allow the onset radius in Lemma~\ref{lem_decomposition_lemma}, Lemma~\ref{lem_good_slices}, Corollary~\ref{cor_quasiflat_is_close_to_cone} and 
Theorem~\ref{thm_cone_is_close_to_quasiflat} to depend on $p$ instead of just $d(p,Q)$, then these results continue to hold for pointed Morse quasiflats with the same proofs. We recall that $\partial_T Q$ is defined in Definition~
\ref{def_Tits_boundary}.

\begin{lem}\label{lem_ideal_equ_tits}
	
Let $Q\subset X$ be a pointed Morse quasiflat. 
Then its ideal boundary and its Tits boundary agree, $\D_T Q=\di Q$. 
\end{lem}

\begin{proof}
Suppose that $(x_k)$ is a sequence in $Q$ such that the geodesic segements $\rho_k$ from $p$ to $x_k$ converge to a geodesic ray $\rho$.
By asymptotic conicality (Theorem~\ref{thm_cone_is_close_to_quasiflat}), for given $\eps>0$ there exists $R_{\text{ac}}>0$ such that $d(x_k',Q)<\eps\cdot R_{\text{ac}}$ where $x_k'$ denotes the point on
$\rho_k$ at distance $R_{\text{ac}}$ from $p$. Since $(x_k')$ converges to the point on $\rho$ at distance $R_{\text{ac}}$ from $p$, we deduce the claim.
\end{proof}

\begin{prop}\label{prop_unique_tangent_cone_infty}
 Let $Q$ be an $L$-Lipschitz $(L,A)$ pointed $n$-dimensional Morse quasiflat represented by $T\in \bZ_{n,loc}(X)$. 
 Then $T$ has a unique tangent cone at infinity. Namely, for any base point $p\in X$ the rescalings $h_{p,\lambda\#}T$ converge with respect to local flat topology to a current $T_{p,0}\in \bZ_{n,loc}(X)$ with the 
 following properties.
 \begin{enumerate}
  \item $T_{p,0}$ is conical with respect to $p$, $h_{p,\lambda\#}T_{p,0}=T_{p,0}$;
  \item $\M(T_{p,0}\on B_r(p))\leq\Theta\cdot r^n$ for $r\geq 0$; 
  \item there exists a function $\delta:[0,\infty)\to[0,\infty)$ with $\lim\limits_{r\to\infty}\delta(r)=0$ depending only on $p,X,L,A,n$ and the Morse data of $Q$ such that for all $R\geq r$ holds
	\[d_H(Q \cap B_p(R),\spt(T_{p,0})\cap B_p(R))\le \delta(r)\cdot R;\] 
  \item  $\di \spt(T_{p,0})=\D_T Q$.
 \end{enumerate}
If $Q$ is a Morse quasiflat, then we can strengthen (3) such that $\delta$ depends on $d(p,Q)$ rather than $p$.
\end{prop}

\begin{proof}
The upper density bound of $T$ implies via compactness (\cite[Theorem 2.3]{higherrank}) that $h_{p,\lambda\#}T$ subconverges in the local flat topology to a current $T_{p,0}\in \bZ_{n,loc}(X)$.
By Corollary~\ref{cor_homotopy} and Theorem~\ref{thm_cone_is_close_to_quasiflat}, we find for every $\eps>0$ a radius $R_0$ such that for all $R_1\geq R_0$ holds 
$\F(T\on B_p(R),C_p(S_{R_1})\on B_p(R))\leq C\cdot \Theta\cdot\eps\cdot R^{n+1}$ for all $R\in [R_0,R_1]$ and a constant $C$ depending only on $L, n$ and $X$. 
Hence for all $\lambda, \lambda'<\frac{r_0}{R_0}$ we obtain 
\[\F(h_{p,\lambda\#}T\on B_p(r_0),h_{p,\lambda'\#}T\on B_p(r_0))\leq 2C\cdot \Theta\cdot\eps\cdot r_0^{n+1}.\] 
Hence $h_{p,\lambda\#}T$ actually converges to $T_{p,0}$ as $\lambda\to 0$. Since  $h_{p,\lambda}\circ h_{p,\lambda'}=h_{p,\lambda\lambda'}$ holds, we see that $T_{p,0}$ is conical, hence (1).
Lower semicontinuity of mass with respect to weak convergence yields (2), the claim on the upper density bound of $T_{p,0}$. (4) follows from (3) and Lemma~\ref{lem_ideal_equ_tits}. We turn to (3).
For every $\lambda\in(0,1)$ we have $\spt(h_{p,\lambda\#}T)\subset h_{p,\lambda}(\spt(T))\subset \ccone_p(\spt(T))$. Hence $\spt(T_{p,0})\subset \ccone_p(\spt(T))$.
On the other hand, by Theorem~\ref{thm_cone_is_close_to_quasiflat}, Corollary~\ref{cor_homotopy} and Corollary~\ref{cor_good_part_covers}, we find for every $\eps>0$
an $R>0$ such that $\spt T\cap B_p(r)\subset N_{\eps r}(\spt(T_{p,0}))$ for all $r\geq R$. Together this shows (3).
\end{proof}

\begin{cor}\label{cor_sublinear_close}
	Let $Q$ be an $n$-dimensional pointed Morse $(L,A)$ quasiflat in a proper metric space $X$ with convex geodesic bicombing.
	There exists a function $\delta:[0,\infty)\to[0,\infty)$ with $\lim\limits_{r\to\infty}\delta(r)=0$ depending only on $p,X,L,A,n$ and the Morse data of $Q$ such that for all $R\geq r$ holds
	\[d_H(Q\cap B_p(R) , C_p(\D_T Q)\cap B_p(R))\le \delta(r)\cdot R.\]
 If $Q$ is a Morse quasiflat, then $\delta$ depends on $d(p,Q)$ rather than $p$.
\end{cor}

\begin{proof}
	As $T_{p,0}$ is conical with respect to $p$, we have $\spt(T_{p,0})\subset C_p(\D_T Q)$ by \cite[Lemma 7.2]{higherrank} and Lemma~\ref{lem_ideal_equ_tits}. But $C_p(\D_T Q)\subset\ccone_p(\spt(T))$ and the claim follows from
	Theorem~\ref{thm_cone_is_close_to_quasiflat} and Proposition~\ref{prop_unique_tangent_cone_infty} (3).
\end{proof}

\begin{thm}
	\label{thm:uniqueness body}
	Suppose $X$ is a proper metric space with convex geodesic bicombing. Let $Q_1$ and $Q_2$ be two Morse quasiflats in $X$. Suppose $\partial_T Q_1=\partial_T Q_2$. 
	Then $d_H(Q_1,Q_2)<C$ where $C<\infty$ depends only on $\dim Q_1$ and the Morse data of $Q_1$ and $Q_2$.
\end{thm}

\begin{proof}
	By Corollary~\ref{cor_sublinear_close}, $Q$ and $Q'$ are at sublinear distance from each other. Now the theorem follows from \cite[Proposition 10.4 and Theorem 9.5]{HKS}.
\end{proof}

\subsection{Cycle at infinity}
\begin{definition}
	\label{morse}
Let $W$ be a topological space.	
For $[\sigma]$ in $H_k(W,\mathbb Z)$, we define the \textit{homological support set} of $[\sigma]$, denoted $S_{[\sigma]}$, to be $\{z\in W\setminus Y\mid i_{\ast}[\sigma]\neq \textmd{Id}\}$, here $i:H_{k}(W,\mathbb Z)\to H_{k}(W,W\setminus\{z\},\mathbb Z)$ is the inclusion homomorphism.

Take $[\sigma]$ in $H_k(W,\mathbb Z)$. The homology class $[\sigma]$ is \emph{immovable} if for any open set $O$ containing such that $S_{[\sigma]}\subset O$, there exist chain $\beta$ and cycle $\gamma$ such that $\textmd{Im}\ \gamma\subset O$ and $\sigma=\partial\beta+\gamma$. 
\end{definition}

Recall that we use $\CT(X)$ to denotes the Tits cone of $X$. Let $\cC(\partial_T Q)$ be the subspace of $\CT(X)$ made of all Hausdorff classes of reparameterization of $\si$-rays which are in $\partial_T Q$ (topologically $\cC(\partial_T Q)$ is a cone over $\partial_T Q$). We equip $\cC(\partial_T Q)$ with the induced metric from $\CT(X)$. Note that when $X$ is CAT$(0)$, $\cC(\partial_T Q)$ is the Euclidean cone over $\partial_T Q$.
\begin{prop}
	\label{prop:full support}
	Let $Q$ be an $n$-dimensional $(L,A)$ pointed Morse quasiflat. Then 
	\begin{enumerate}
		\item $\cC(\partial_T Q)$ is bilipschitzly homeomorphic to $\mathbb E^n$;
		\item the map 
		$ H_{n-1}(\partial_T Q,\D_T Q\setminus\{p\},\mathbb Z)\to  H_{n-1}(\partial_T X,\partial_T X\setminus\{p\},\mathbb Z)$ is injective for each point $p\in\D_T Q$;
		\item $\D_TQ$ is the homological support set of some immovable class $[\si]\in \tilde H_{n-1}(\D_T X)$.
	\end{enumerate}
	
\end{prop}

This generalizes the fact that pointed Morse quasi-geodesics give rise to isolated points in the Tits boundary.

\begin{proof}
Let $(X_\omega,p_\omega)$ be an asymptotic cone of $X$ with fixed base point $p\in X$ and denote by $Q_\om\subset X_\om$ the ultralimit of $Q$. Let $\CT(X)$ be the Tits cone of $X$ with cone point $o$ and let $i:\CT(X)\to X_\om$ be the canonical isometric embedding as in \cite[Lemma 10.6]{kleiner1999local} such that $i(o)=p_\om$. The map $i$ sends $C_{o}(\D_T Q)$ to the cone $C_{p_\om}(\D_T Q)\subset X_\om$.
Corollary~\ref{cor_sublinear_close} implies $C_{p_\om}(\D_T Q)=Q_\om$. Thus $C_{p_\om}(\D_T Q)$ is bilipschitz to $\mathbb E^n$ and (1) holds. 
Since $Q$ is Morse, \[H_n(C_{p_\om}(\D_T Q),C_{p_\om}(\D_T Q)\setminus\{q\},\mathbb Z)\to H_n(X_\om,X_\om\setminus\{q\},\mathbb Z)\] is injective for each $q\in C_{p_\om}(\D_T Q)$. Hence we deduce the injectivity of \[H_n(C_{o}(\D_T Q),C_{o}(\D_T Q)\setminus\{q\},\mathbb Z)\to H_n(\CT(X),\CT(X)\setminus\{q\},\mathbb Z)\] for each for each $q\in C_{o}(\D_T Q)$.
Now (2) follows from the Künneth formula (cf. \cite[pp. 190, Proposition 2.6]{dold2012lectures}). Consider the following commuting diagram
there is a commutative diagram
$$\begin{array}{ccc}
  H_n(C_{o}(\D_T Q),C_{o}(\D_T Q)\setminus\{q\},\mathbb Z) & \rightarrow  & H_n(\CT(X),\CT(X)\setminus\{q\},\mathbb Z) \\
   \downarrow   &   & \downarrow \\
     \tilde H_{n-1}(\D_T Q,\mathbb Z)&  \rightarrow       &  \tilde H_{n-1}(\D_T X,\mathbb Z)
\end{array}$$
where the two downward arrows are isomorphisms. The fundamental class of $C_{o}(\D_T Q)$ gives rise to $[\si]\in H_{n-1}(\D_T X,\mathbb Z)$ under the diagram. Then $[\si]$ can be represented by a singular cycle whose image is in $\D_T Q$. Moreover, (2) implies $S_{[\si]}=\D_T Q$. Thus (3) follows.
\end{proof}

Besides treating $\D_T Q$ as the homological support set of some class, $\D_T Q$ can be alternatively interpreted as the support set of some integral current as follows. 

For a current $T\in\I_{n,loc}(X)$ we define its {\em density at infinity} by 
\[\Theta_\infty(T)=\limsup\limits_{r\to\infty}\frac{\|T\|(B_p(r))}{r^n}.\]

\begin{cor}\label{cor_mass_of_Tits}
	Let $Q$ be an $n$-dimensional pointed Morse quasiflat  represented by $T\in \bZ_{n,loc}(X)$. Let $T_{p,0}$ be as in Proposition~\ref{prop_unique_tangent_cone_infty}. 
	Then there exists a cycle $\si\in\bZ_{n-1,c}(\D_T X)$ with  
	$\spt(\si)=\D_T Q$.
\end{cor}

\begin{proof}
	It follows from Theorem~\ref{thm_cone_is_close_to_quasiflat}, that $\D_T Q\subset\D_T X$ is compact. By Proposition~\ref{prop_unique_tangent_cone_infty}, $T_{p,0}$
	is conical and \cite[Theorem 9.3]{higherrank} provides a cycle $\si\in\bZ_{n-1,c}(\D_T X)$ with  
	$\spt(\si)=\di Q$ (again, \cite[Theorem 9.3]{higherrank} does not depend on the rank assumption in \cite{higherrank}, see the paragraph in \cite{higherrank} before \cite[Theorem 9.3]{higherrank}).
	Lemma~\ref{lem_ideal_equ_tits} completes the proof.
\end{proof}

\begin{remark}
Proposition~\ref{prop:full support} (3) and Corollary~\ref{cor_mass_of_Tits} are compatible in the sense that by the proof of Proposition~\ref{prop:full support}, the class in Proposition~\ref{prop:full support} (3) can be represented by a Lipschitz cycle with its image contained in $\D_T Q$. Then the integral current associated with this Lipschitz cycle is $\si$ in Corollary~\ref{cor_mass_of_Tits}.
\end{remark}

\begin{remark}
	\label{rmk:sphere}
It is natural to ask whether $\D_T Q$ is homeomorphic, or bilipschitz to the standard sphere. This is not a direct consequence of Proposition~\ref{prop:full support} (1). By \cite{siebenmann1979complexes}, if $\D_T Q$ is bilipschitz to a piecewise Euclidean simplicial complex, then $\D_T Q$ is homeomorphic to a sphere. This holds, e.g. when $X$ is a $\cat(0)$ cube complex \cite{boundary}. 
\end{remark}

Let $X$ be a proper metric space with convex geodesic bicombing. 
Let $\mathbf {Z}_{MQ}(\D_TX)$ be the collection of cycles arising from $Q$ (cf. Corollary \ref{cor_mass_of_Tits}), with $Q$ ranging over all possible pointed Morse quasiflats in $X$. The following is a consequence of the fact that quasi-isometries between metric spaces with convex geodesic bicombing send (pointed) Morse quasiflats to (pointed) Morse quasiflats \cite[Proposition 6.27]{HKS}.

\begin{cor}
	\label{cor_boundary_map}
Let $q:X\to Y$ be a quasi-isometry between two proper metric spaces with convex geodesic bicombing. Then $q$ induces a bijection $$q_*:\mathbf {Z}_{MQ}(\D_TX)\to \mathbf {Z}_{MQ}(\D_TY).$$ 
\end{cor}

\subsection{Remark on Hausdorff distance to a cone}
We refer to Definition~\ref{def_Tits_boundary} for the Tits boundary $\partial_T Q$ of a subset $Q$ of $X$, and to Definition~\ref{def_exp} for the definition of exponential map $\exp_p:\CT(X)\to X$.
\begin{prop}
	\label{prop_hausdorff}
Let $Q$ be an $n$-dimensional Morse quasiflat in a proper metric space $X$ with convex geodesic bicombing. Suppose the restriction  $\exp_p:\cC(\partial_T Q)\to X$ is a quasi-isometric embedding for some (hence any) base point $p\in X$. Then $d_H(C_p(\partial _T Q),Q)<\infty$ for some (hence any) base point $p\in X$.
\end{prop}
\begin{proof}
We deduce from  Corollary~\ref{cor_sublinear_close} that
	\begin{equation*}
\label{eq_linear_div}
\lim_{r\to\infty} \frac{d_H(B_p(r)\cap Q,B_p(r)\cap C_p(\partial _T Q))}{r}=0
\end{equation*}
As $\cC(\partial_T Q)$ is bilipschitz to $\mathbb E^n$ (Proposition~\ref{prop:full support} (1)), we know $C_p(\partial _T Q)$ is a quasiflat in $X$. Thus \cite[Proposition 10.4]{HKS} and \cite[Proposition 9.10]{HKS} imply that $d_H(C_p(\partial _T Q),Q)<\infty$.
\end{proof}

We now give examples where the conclusion of Proposition~\ref{prop_hausdorff} either fails (Example~\ref{example_exp}) or holds (Example~\ref{example_positive}).

\begin{example}
	\label{example_exp}
Let $X$ be $\mathbb E^2$ with the Euclidean metric. Let $A=\{(x,y)\in X\mid y\ge e^x\}$. We glue two copies of $X$ along the convex subset $A$, and denote the resulting $\cat(0)$ space by $Y$. Note that $\partial_T A$ is an arc of length $\pi/2$, and $\partial_T Y$ is obtained by gluing two copies of a circle of length $2\pi$ along the arc of length $\pi/2$ corresponding to $\partial_T A$.

Let $Q$ be $Y$ with the interior of $A$ removed. We claim $Q$ is a 2-dimensional Morse quasiflat in $Y$ which is not at finite Hausdorff distance from any geodesic cone over a subset of $\partial_T X$. Indeed, $Q$ is a quasiflat because $Q$ is a union of two pieces, each of them is bilipschitz to a Euclidean half plane. $Q$ is Morse because it is top-dimensional. Because of Corollary~\ref{cor_quasiflat_is_close_to_cone}, showing $Q$ is not Hausdorff close to any geodesic cone reduces to showing $d_H(Q,C_p(\partial_T Q))=\infty$. Note that
$\partial_T Q$ is a circle of length $3\pi$ obtained from $\partial_T Y$ by removing the interior points of $\partial_T A$. Thus $d_H(Q,C_p(\partial_T Q))=\infty$ by construction.

It is worth noting that the configuration in this example also arises naturally in the study of group theory, notably in the recent work of Lamy and Przytycki \cite{lamy2018presqu} where they define a ``generalized building'' acted upon by tame automorphism groups.
\end{example}

\begin{example}
	\label{example_positive}
Let $X$ be a symmetric space of non-compact type or a Euclidean building, and let $Q\subset X$ be a Morse quasiflat. Then $\partial_T Q$ is the support set of a top-dimensional cycle in $\partial_T X$. Thus $\partial_T Q$ is a union of Weyl chambers, each of which is the Tits boundary of a Weyl cone in $X$. The exponential map $\exp_p:\cC(\partial_T Q)\to X$ is a quasi-isometric embedding, and the above proposition implies that $Q$ is at finite Hausdorff distance from a finite union of Weyl cones. This recovers a key result in \cite{kleiner1997rigidity,eskin1997quasi}. A similar discussion applies when $X$ is a $\cat(0)$ cube complex. See Theorem~\ref{thm:cube Morse lemma}. Other examples where Proposition~\ref{prop_hausdorff} applies include 2-dimensional $\cat(0)$ complexes with finite shape as in \cite{xie2005tits,bks1}, universal covers of closed 4-dimensional manifolds with non-positively curved analytic Riemannian metric \cite{hummel1998tits} .
\end{example}

\section{Morse quasiflats of Euclidean growth in CAT(0) spaces}
\label{sec_rigidity}

The following can be shown similarly to \cite[Lemma A.11]{huang_quasiflat}.

\begin{lemma}[geodesic extension property]\label{lem_geo_extension}
Let $Y$ be a CAT(0) space and $P\subset Y$ a bilipschitz flat of full support. Then every geodesic segment $\rho:[0,1]\to Y$ with $\rho(1)\in P$
extends to a  geodesic ray $\hat\rho:[0,\infty)\to Y$ with $\hat\rho([1,\infty))\subset P$.
\end{lemma}

We denote by $\mathcal H^n$ the $n$-dimensional Hausdorff measure. For a subset $A\subset X$ of a metric space $X$ we define its $n$-dimensional Hausdorff
volume growth as $\Theta_n(A)=\limsup\limits_{r\to\infty}\frac{\mathcal H^n(A\cap B_p(r))}{r^n}$. Hence  $\Theta_n(\E^n)=\omega_n$, the volume of the unit ball in $\E^n$.

\begin{lemma}\label{lem_bilip_CAT(0)}
Let $Z$ be a CAT(0) space which is bilipschitz to $\E^n$. 
Suppose that its volume growth is at most Euclidean, $\Theta_n(Z)\leq \omega_n$. Then $Z$ is isometric to $\E^n$.
\end{lemma}

\begin{proof}
It is enough to show that $\D_T Z$ is a round sphere.
After possibly passing to an asymptotic cone, we may assume that $Z$ itself is a Euclidean cone
over $\D_T Z$ with tip $o$.
Since $Z$ is bilipschitz to $\E^n$,  the link $\Si_p Z$ is isometric to  a round $(n-1)$-sphere for almost all points $p$. Since $Z$ is a Euclidean cone,
all but possibly the tip has round $(n-1)$-spheres as links. For every point $p\in Z$ we obtain a map $f_p:\Si_p Z\to \D_T Z$ which is distance nondecreasing by choosing
a geodesic ray for each direction.
From the Euclidean growth assumption and the CAT(0) property, it follows that the image of $f_p$ has full $\mathcal H^{n-1}$-measure in
$\D_T Z$. Since $Z$ is a Euclidean cone bilipschitz to $\E^n$, the $\mathcal H^{n-1}$-measure of a ball in $\D_T Z$ is positive. Therefore,
the image of $f_p$ is actually dense. 

Now we show that each point in $\D_T Z$ has a unique antipode. Consider a geodesic ray $\rho$ with $\rho(0)=o$ and $[\rho]=\xi\in\D_T Z$. Set $p_k=\rho(k)$ and $f_k=f_{p_k}$. Take an ultralimit of the $f_k$
to obtain a distance nondecreasing map $f_\infty:\mathbb S^{n-1}\to\D_T Z$. The image of $f_\infty$ is complete and therefore $f_\infty$
is onto. Note that for $\zeta\in\D_T Z$ holds $\lim\limits_{k\to \infty}\angle_{p_k}(\xi,\zeta)=\angle_T(\xi,\zeta)$ where $\angle_T$ denotes the Tits angle. Hence, if $(\zeta_k)$ is a sequence in $\D_T Z$ with $\om\lim_{k\to\infty} \zeta_k=\zeta$, then
$\om\lim_{k\to \infty}\angle_{p_k}(\xi,\zeta_k)=\angle_T(\xi,\zeta)$. This shows that
$f_\infty$ sends the north pole to $\xi$ and preserves the distance to the north pole. 
Therefore $\xi$ has a unique antipode, the image of the south pole.

As a consequence, any two lines in $Z$ which are one-sided asymptotic are actually parallel. This shows that $Z$ splits isometrically as $Z\cong l\times l^\perp$
for every line $l\subset Z$. Hence $Z$ is isometric to $\E^n$.
\end{proof}

\begin{lemma}\label{lem_euc_vol}
Let $Y$ be a CAT(0) space and $P\subset Y$ an $n$-bilipschitz flat of full support. 
Suppose that the volume growth of $P$ is at most Euclidean, $\Theta_n(P)\leq \omega_n$. 
Then $P$ is a flat in $Y$.
\end{lemma}

\begin{proof}
Since $P$ is a bilipschitz flat, it has links $\Si_p P\subset \Si_p Y$   which are
round $(n-1)$-spheres at almost all points. From the geodesic extension property, we see that for every $v\in \Si_p P$ there exists a geodesic ray $\rho$ with $\dot\rho(0)=v$ and which lies entirely in $P$. The Euclidean growth assumption and the CAT(0) property imply 
$\mathcal H^n(P\cap B_p(r))\equiv\omega_n r^n$ and that every point in $P\cap B_p(r)$ can be joined to $p$ by a geodesic lying in $P$. In particular,
$P$ is convex and the claim follows from Lemma~\ref{lem_bilip_CAT(0)}.
\end{proof}

\begin{thm}\label{thm_growth_rigidity body}
 Let $X$ be a proper CAT(0) space. Let $Q\subset X$ be an $n$-dimensional Morse quasiflat represented by a current $T\in\bZ_{n,loc}(X)$.
Suppose that the density at infinity of $T$ is at most Euclidean, $\Theta_\infty(T)\leq \omega_n$.
 Then there exists an $n$-flat $F\subset X$ such that  $d_H(Q,F)<C$ where $C$ depends only on $L, A, n, X$ and the Morse data of $Q$.
\end{thm}

\begin{proof}
By Proposition~\ref{prop_unique_tangent_cone_infty}, we see $\Theta_\infty(T_{p,0})\leq \Theta_\infty(T)$. Hence Corollary~\ref{cor_mass_of_Tits} provides a cycle $\si\in\bZ_{n-1,c}(\D_T X)$ with 
$\M(\si)\leq\mathcal H^{n-1}(\mathbb S^{n-1})$ and $\spt(\si)=\D_T Q$. In particular, $\mathcal H^{n-1}(\D_T Q)\leq\mathcal H^{n-1}(\mathbb S^{n-1})$.

Let $Q_\om$ be an ultralimit of $Q$ in an asymptotic cone $X_\om$ of $X$. By Theorem~\ref{thm_cone_is_close_to_quasiflat}, $Q_\om$ is isometric to a Euclidean cone over $\D_T Q$.
From the Morse property, we know that $Q_\om$ is a bilipschitz flat of full support, and by the above estimate, $Q_\om$ has at most Euclidean volume growth in
$X_\om$. Lemma~\ref{lem_euc_vol} implies that $Q_\om$ is a flat.  Hence $\D_T Q$
is isometric to a round $(n-1)$-sphere.
 From Proposition~\ref{prop:full support} we know that $\partial_T Q$ does not bound a hemissphere in $\partial_T X$.
 Hence \cite[Proposition 2.1]{leeb_rigidity} implies that there is an $n$-flat $F\subset X$ with $\partial_T F=\partial_T Q$.  It follows from \cite[Proposition 10.4 and Theorem 9.5]{HKS} that $Q$ is at uniformly finite Hausdorff distance from $F$.
\end{proof}

\appendix

\section{Some examples of quasi-isometric classification}
\label{sec_appendix}
The goal of this appendix is to prove Corollary~\ref{cor_qi_classification}.

\subsection{Background on graph products}
Let $\Gamma$ be a finite simplicial graph. For each vertex $v\in\Gamma$, we associated a vertex group $G_v$. Let $G_{\Gamma}$ be the graph products of the $G_v$'s over $\Gamma$. Each full subgraph $\Gamma'\subset\Gamma$ induces an embedding $G_{\Gamma'}\to G_{\Gamma}$, which gives a \emph{standard subgroup} of $G_\Gamma$. If $\Gamma'$ is a (maximal) complete subgraph of $\Gamma$, then $G_{\Gamma'}$ is the product of its vertex groups, hence we call $G_{\Gamma'}$ a \emph{(maximal) standard product subgroup} of $G_\Gamma$ (when $\Gamma'=\emptyset$, $G_{\Gamma'}$ is the identity subgroup, which is also treated as a standard product subgroup).
The left cosets of a standard (product) subgroups are called  \emph{standard (product) cosets}. 
The \emph{rank} of a standard coset $gG_{\Gamma'}$ is defined to be the cardinality of vertices in $\Gamma'$. Two standard cosets $g_1G_{\Gamma_1}$ and $g_2G_{\Gamma_2}$ are \emph{parallel} if $\Gamma_1=\Gamma_2$ and $g^{-1}_2g_1$ are contained in a standard subgroup $G_{\Gamma'}$ of form $G_{\Gamma'}=G_{\Gamma_1}\times G_{\Gamma'_1}$ (in particular $\Gamma'=\Gamma_1\circ\Gamma'_1$). Parallelism between standard cosets forms an equivalence relationship.

Define the \emph{extension graph} of $G_\Gamma$ \cite{kim2013embedability}, denoted by $E(G_\Gamma)$, as follows. Each vertex of $E(G_\Gamma)$ corresponds to a parallel class of rank 1 standard cosets. Two vertices of $E(G_\Gamma)$ are joined by an edge if there exist representatives $g_1G_{v_1},g_2G_{v_2}$ from these two parallel classes such that they are contained in a common standard product coset of rank 2 (in particular $v_1$ and $v_2$ are adjacent). Note that complete subgraphs of $E(G_\Gamma)$ with $k$ vertices are in 1-1 correspondence with parallel classes of standard product cosets of rank $k$, and maximal complete subgraphs of $E(G_\Gamma)$ are in 1-1 correspondence with maximal standard product cosets of $G_\Gamma$.

Define the \emph{right-angled building} of $G_\Gamma$ \cite{davis1994buildings}, denoted by $B(G_\Gamma)$, as follows. The collection of all standard product cosets form a poset $\mathcal P$ under inclusion. Each interval in this poset is a Boolean lattice. $B(G_\Gamma)$ is a cube complex whose $0$-skeleton is can be identified with $\mathcal P$. Cubes of $B(G_\Gamma)$ correspond to intervals in $\mathcal P$. The rank of a vertex in $B(G_\Gamma)$ is defined to be the rank of the associated standard coset.

Let $G_\Gamma$ and $H_\Gamma$ be two graph products such that their vertex groups $G_v$ and $H_v$ are countably infinite. Then there exists an isomorphism of cube complexes $i:B(G_\Gamma)\to B(H_\Gamma)$ which preserves rank of vertices (see \cite[Proposition 1.2]{haglund2003constructions}, or \cite[Lemma 3.15 and Corollary 5.23]{huang2018groups}). Restricting $i$ to rank 0 vertices, we obtain a bijection $i':G_\Gamma\to H_\Gamma$ sending standard product cosets to standard product cosets. Thus $i'$ maps parallel rank 1 standard cosets to parallel rank 1 standard posets. Then $i$ induces an isomorphism $i_*:E(G_\Gamma)\to E(H_\Gamma)$.

We define $\Gamma$ to be \emph{rigid} if the associated right-angled Artin group $A_\Gamma$ has finite outer automorphism group. Any atomic graph is rigid.  
\begin{prop}
	\label{prop_combinatorial_rigidity}
Let $\Gamma_1$ and $\Gamma_2$ be finite simplicial graphs. Let $G_{\Gamma_1}$ and $H_{\Gamma_2}$ be two graph products of finitely generated infinite groups. Suppose $\Gamma_1$ and $\Gamma_2$ are rigid. Suppose there exists a quasi-isometry $q:G_{\Gamma_1}\to H_{\Gamma_2}$ and $C>0$ such that both $q$ and its quasi-isometric inverse map maximal standard product cosets to maximal standard products cosets up to Hausdorff distance $C$ ($C$ does not depend on the cosets). Then there exists a graph isomorphism $f:\Gamma_1\to \Gamma_2$ such that $G_{v}$ is quasi-isometric to $H_{f(v)}$ for any $v\in\Gamma_1$.
\end{prop}

\begin{proof}
In the proof we will use the vocabulary of coarse containment and coarse intersection from \cite[Section 2.1]{mosher2004quasi}.	
We start with the observation that two standard cosets $g_1G_{\Gamma'_1}$ and $g_2G_{\Gamma'_2}$ are parallel if and only if they have finite Hausdorff distance. The only if direction is clear. To see the if direction, note that $d_H(g_1G_{\Gamma'_1}g^{-1}_1,g_2G_{\Gamma'_2}g^{-1}_2)<\infty$, then $g_1G_{\Gamma'_1}g^{-1}_1\cap g_2G_{\Gamma'_2}g^{-1}_2$ is finite index in both $g_1G_{\Gamma'_1}g^{-1}_1$ and $g_2G_{\Gamma'_2}g^{-1}_2$ by \cite[Corollary 2.4]{mosher2004quasi}. As each vertex group is infinite, \cite[Corollary 3.6]{antolin2015tits} implies that $g_1G_{\Gamma'_1}g^{-1}_1=g_2G_{\Gamma'_2}g^{-1}_2$. Now \cite[Lemma 3.2]{antolin2015tits} implies that $g_1G_{\Gamma'_1}$ and $g_2G_{\Gamma'_2}$ are parallel.
	
As $\Gamma_1$ and $\Gamma_2$ are rigid, each vertex is an intersection of maximal complete subgraphs. Thus each rank 1 standard subgroups is an intersection of maximal standard product subgroups. By \cite[Lemma 2.2]{mosher2004quasi}, $q$ maps rank 1 standard cosets to rank 1 standard cosets up to finite Hausdorff distance. By the observation in the previous paragraph, we know $q$ induces a bijection of the vertex sets $q_*:E^{(0)}(G_{\Gamma_1})\to E^{(0)}(H_{\Gamma_2})$. We deduce from \cite[Corollary 2.4]{mosher2004quasi} and \cite[Section 3]{antolin2015tits} that two rank 1 standard cosets correspondence to adjacent vertices in the extension graph if and only if they are coarsely contained in a common maximal standard product coset. Thus $q_*$ extends to a graph isomorphism $q_*:E(G_{\Gamma_1})\to E(H_{\Gamma_2})$.
	
Let $A_{\Gamma_i}$ be the right-angled Artin group defined on $\Gamma_i$. As discussed before, there are rank preserving isomorphisms $i_1:B(A_{\Gamma_1})\to B(G_{\Gamma_1})$ and $i_2:B(H_{\Gamma_2})\to B(A_{\Gamma_2})$, with induced isomorphisms $(i_1)_*:E(A_{\Gamma_1})\to E(G_{\Gamma_1})$ and $(i_2)_*:E(H_{\Gamma_2})\to E(A_{\Gamma_2})$. Let $i_*=(i_2)_*\circ q_*\circ(i_1)_*$. Then \cite[Lemma 4.12]{huang2017quasi} implies $i_*$ is induced by a rank preserving isomorphism $i:B(A_{\Gamma_1})\to B(A_{\Gamma_2})$. Hence $q_*$ is induced by a rank preserving isomorphism $q':B(G_{\Gamma_1})\to B(H_{\Gamma_2})$. Then $q'$ sends rank 0 vertices to rank 0 vertices, and links of rank 0 vertices in $B(G_{\Gamma_1})$ (resp. $B(H_{\Gamma_2}$) is $\Gamma_1$ (resp. $\Gamma_2$). Thus $\Gamma_1$ and $\Gamma_2$ are isomorphic. We restrict $q'$ to rank 0 vertices to obtain bijection $q'':G_{\Gamma_1}\to H_{\Gamma_2}$. For any maximal standard product coset $F\subset G_{\Gamma_1}$, it follows from the construction of $q''$ that $q''(F)$ and $q(F)$ has finite Hausdorff distance. As every point in $G_{\Gamma_1}$ is the intersection of maximal standard product cosets containing this point, by \cite[Lemma 2.2]{mosher2004quasi} $q''$ is at a uniform bounded distance from $q$. In particular $q''$ is a quasi-isometry and the proposition follows.
\end{proof}

\subsection{Atomic graph products}
We start with a simple observation. For a flat $F$ in a CAT(0) cube complex, we denote the collection of hyperplanes intersecting $F$ transversally by $\mathcal H(F)$.
\begin{lemma}
	\label{lem_uniform}
Let $X$ be a CAT(0) cube complex. Let $\{F_\lambda\}_{\lambda\in\Lambda}$ be a family of $k$-dimensional flats which are uniformly Morse in the sense of Definition~\ref{def_morse}. Then there exists $C>0$ such that for any $F_\lambda$ and convex subcomplex $K$ with $\mathcal H(F_\lambda)\subset\mathcal H(K)$, then $F_{\lambda}\subset N_C(K)$.
\end{lemma}

\begin{corollary}
	\label{cor_qi_classification}
Let $\Gamma_1$ and $\Gamma_2$ be atomic graphs. Let $G_{\Gamma_1}$ and $H_{\Gamma_2}$ be two graph products with vertex groups being infinite right-angled Coxeter groups containing rank 1 elements. Then $G_{\Gamma_1}$ and $H_{\Gamma_2}$ are quasi-isometric if and only if there exists a graph isomorphism $f:\Gamma_1\to \Gamma_2$ such that $G_{v}$ is quasi-isometric to $H_{f(v)}$ for any $v\in\Gamma_1$.
\end{corollary}

\begin{proof}
Note that if $G_\Gamma$ and $H_\Gamma$ are graph products with the same defining graph $\Gamma$ and there are bilipschitz maps $f_v:G_v\to H_v$ between each vertex groups, then Green's normal form theorem for graph products \cite[Theorem 3.9]{green1990graph} (see also \cite[Theorem 2.2]{antolin2015tits}) implies there is a well-defined map $f:G_\Gamma\to H_\Gamma$ induced by $\{f_v\}_{v\in V\Gamma}$, which is bilipschitz. Now the if direction follows as two right-angled Coxeter groups are quasi-isometric if and only if they are bilipschitz  \cite{whyte1999amenability}.
	
Let $q:G_{\Gamma_1}\to H_{\Gamma_2}$ be a quasi-isometry. By Proposition~\ref{prop_combinatorial_rigidity}, it suffices show $q$ preserves maximal standard product cosets up to uniform finite Hausdorff distance. Define \emph{star coset} of $G_{\Gamma_1}$ to be a left coset of a subgroup of form $G_{\st(v)}$ for some vertex $v\in\Gamma_1$. As each maximal standard product coset can be realized as the coarse intersection of star cosets, it suffices to show $q$ preserves star cosets up to uniform finite Hausdorff distance. Actually, we are reduced to show the claim that any star coset $C$, $q(C)$ is contained in a uniform neighborhood of another star coset. This reduction follows by considering the quasi-isometry inverse, and noting that a star coset is not coarsely contained in a different star coset (as $\Gamma$ is atomic and each vertex group is infinite).

Suppose $G_{\Gamma_1}$ (resp. $H_{\Gamma_2}$) is a right-angled Coxeter groups with defining graph $\Lambda_1$ (resp. $\Lambda_2$). There is a natural map $p_1:\Lambda_1\to\Gamma_1$ sending the defining graphs of vertex groups to the corresponding vertex in $\Gamma_1$. Thus if $\Lambda\circ\Lambda'$ is a join subgraph of $\Lambda_1$, then each vertex of $p_1(\Lambda)$ has distance $\leq 1$ to every vertex in $p_1(\Lambda')$. Hence either $p_1(\Lambda)=p_1(\Lambda')$ and $p_1(\Lambda)$ is contained in a star of a vertex of $\Gamma_1$, or $p_1(\Lambda)\neq p_1(\Lambda')$, in which case $p_1(\Lambda\circ\Lambda')$ splits as join of two subgraphs, hence is again contained in a vertex star of $\Gamma_1$ as $\Gamma_1$ does not have cycle of length $<5$.

Let $X_1$ and $X_2$ be the Davis complexes associated to $G_{\Gamma_1}$ and $H_{\Gamma_2}$. A \emph{standard subcomplex} of $X_1$ is a subcomplex arising from a full subgraph of $\Lambda_1$. To prove the above claim, it suffices to show $q(C)$ is coarsely contained in a uniform neighborhood of a convex subcomplex $K$ which splits as a nontrivial product of two CAT(0) cube complexes $K=K_1\times K_2$. The reason is that such $K$ is contained in a standard subcomplex $K'$ which splits as a product of two standard subcomplexes $K'_1\times K'_2$, and any such $K'$ is contained in a star coset by the previous paragraph.

Let $S$ be a star coset. Suppose $S=G_{\st(v)}$ for $v\in\Gamma_1$. Let $\ell_1$ (resp. $\ell_2$) be a rank 1 periodic geodesic in the standard subcomplex associated with $G_v$ (resp. $G_{lk(v)}$). For $g_1\in G_v$ and $g_2\in G_{lk(v)}$, let $F_{g_1,g_2}=g_1\ell_1\times g_2\ell_2$. Then $F_{g_1,g_2}$ is a Morse flat \cite[Corollary 1.20]{HKS}. For each star coset of $G$, we select a family of Morse flat in a similar way. The collection of all such Morse flat is uniformly Morse, as there are only finitely many orbits of them under the isometry group action. Thus they have a common Morse data.
By Theorem~\ref{thm:morse lemma}, there exists $D>0$ such that for any $g_1,g_2$ as above, there is a Morse flat $F'_{g_1,g_2}$ such that $d_H(q(F_{g_1,g_2}),F'_{g_1,g_2})<D$. 
As $d_H(S,\cup_{g_1\in G_v,g_2\in G_{lk(v)}}F_{g_1,g_2})<D'$ for some $D'$ independent of the star coset, it remains to show that there is $D''>0$ independent of the star coset such that $\cup_{g_1\in G_v,g_2\in G_{lk(v)}}F'_{g_1,g_2}$ is contained in the $D''$-neighborhood of a convex subcomplex $K\subset X_2$ admitting a nontrivial product splitting. We will only show this when there exist $g_1,g'_1\in G_v$ and $g_2,g'_2\in G_{lk(v)}$ such that $d_H(g_2\ell_2,g'_2\ell_2)=\infty$ and $d_H(g_1\ell_1,g'_1\ell_1)=\infty$. 
Other cases are simpler and similar. Note that the existence of such $g_1,g'_1,g_2,g'_2$ implies that by taking possibly different $g_1,g'_1,g_2,g'_2$, we can assume $\D_\infty g_2\ell_2\cap \D_\infty g'_2\ell_2=\emptyset$ and $\D_\infty g_1\ell_1\cap \D_\infty g'_1\ell_1=\emptyset$.

Let $\mathcal{H}_{g_1,g_2}=\mathcal H(F'_{g_1,g_2})$. These hyperplanes intersect $F'_{g_1,g_2}$ in parallel family of lines. Since $F'_{g_1,g_2}$ is Morse, by \cite[Proposition 11.3]{HKS} and an argument similar to \cite[Theorem 3.4]{huang2016cocompactly}, we know there are only two parallel family of lines which are orthogonal. 

Note that the coarse intersection of $F_{g_1,g_2}$ and $F_{g_1,g'_2}$ is a geodesic line. Then the same is true for $F'_{g_1,g_2}$ and $F'_{g_1,g'_2}$. Thus $q(g_1\ell_1)$ is Hausdorff close to a geodesic line. Similarly by consider the coarse intersection of $F'_{g'_1,g_2}$ and $F'_{g_1,g_2}$, we know $q(g_2\ell_2)$ is Hausdorff close to a geodesic line. Let $\mathcal H_1$ (resp. $\mathcal H_2$) be the collection of all hyperplanes of $X_2$ which have transversal intersection with a geodesic line that is Hausdorff close to $g_1\ell_1$ (resp. $g_2\ell_2$) for some $g_1\in G_v$ (resp. $g_2\in G_{lk(v)}$). Note that if a hyperplane interests a geodesic line transversely, then it intersects all geodesic lines in the same parallel family transversely. Thus $\mathcal H_i$ is well-defined. For $i=1,2$, let $h_i\in\mathcal H_i$ be a hyperplane dual to a line Hausdorff close to $q(g_i\ell_i)$, then $h_1$ and $h_2$ intersect $F'_{g_1,g_2}$ in orthogonal lines. Thus each element in $\mathcal H_1$ intersects every element in $\mathcal H_2$. This gives rise to a convex subcomplex $K\subset X_2$ with nontrivial product splitting $K=K_1\times K_2$ (the collection of hyperplanes dual to $K_i$ contains $\mathcal H_i$ as a possibly proper subset). Note that $\mathcal{H}_{g_1,g_2}\subset \mathcal H_1\cup\mathcal H_2$ for each $F'_{g_1,g_2}$. By Lemma~\ref{lem_uniform}, $\cup_{g_1\in G_v,g_2\in G_{lk(v)}}F'_{g_1,g_2}$ is coarsely contained in $K$ with uniform constant $D''$. 
\end{proof}


\begin{remark}[Comments on generalizations]
	\label{rmk_graph_product}
Recall that a simplicial graph $\Gamma$ is rigid if the associated right-angled Artin group has finite outer automorphism group. Motivated by quasi-isometric classification of graph products of $\mathbb Z$'s over rigid defining graphs in \cite{bks2,huang2017quasi} and the above corollary, we speculate that the following might be true. Let $G_1$ and $G_2$ be graph products of finitely generated groups with non-trivial Morse boundary (e.g. acylindrical hyperbolic groups) over rigid defining graphs. If $G_1$ and $G_2$ are quasi-isometric, then their defining graphs are isomorphic and the corresponding vertex groups are quasi-isometric. 

To prove this, it suffices to show product subgroups corresponding to maximal cliques in the defining graphs are preserved by quasi-isometries (see Proposition~\ref{prop_combinatorial_rigidity}). 
These subgroups are unions of Morse quasiflats which are products of Morse quasi-geodesics in their factors. It suffices to control the quasi-isometric image of these Morse quasiflats. For this purpose, one can use Theorem~\ref{thm_sublinear_close_intro} if the vertex groups are also CAT(0) (as graph products of CAT(0) groups are CAT(0) \cite[Theorem 8.8]{huang2016groups}). If the vertex groups are coarse median instead, then a Morse version of the main result in \cite{bowditch2019quasiflats} might be helpful.
\end{remark}

\bibliography{morse_quasiflats}
\bibliographystyle{alpha}

\end{document}